\documentclass[12pt]{amsart}
\usepackage[all]{xy}
\usepackage{amssymb}
\usepackage{latexsym}
\usepackage{mathrsfs}
\usepackage{hyperref}
\usepackage{enumerate}
\usepackage{fullpage}

\numberwithin{equation}{section}

\newtheorem{mytheorem}{Theorem}
\numberwithin{mytheorem}{section}
\newtheorem{mylemma}[mytheorem]{Lemma}

\newtheorem{myproposition}[mytheorem]{Proposition}
\newtheorem{mydefinition}[mytheorem]{Definition}

\newtheorem{myremark}[mytheorem]{Remark}

\newcommand{\ab}{\ensuremath{\mathrm{ab}}}

\newcommand{\alg}{\ensuremath{\mathrm{alg}}}

\newcommand{\AT}{\ensuremath{\mathrm{AT}}}
\newcommand{\Art}{\ensuremath{\mathrm{Art}}}

\newcommand{\Aut}{\ensuremath{\mathrm{Aut}}}

\newcommand{\cl}{\ensuremath{\mathrm{cl}}}

\newcommand{\cosoc}{\ensuremath{\mathrm{cosoc}}}

\newcommand{\End}{\ensuremath{\mathrm{End}}}
\newcommand{\et}{\ensuremath{\mathrm{\acute{e}t}}}
\newcommand{\Eul}{\ensuremath{\mathrm{Eul}}}
\newcommand{\ev}{\ensuremath{\mathrm{ev}}}

\newcommand{\Fr}{\ensuremath{\mathrm{Fr}}}
\newcommand{\Frss}{\ensuremath{\mathrm{Fr\text{-}ss}}}

\newcommand{\Gal}{\ensuremath{\mathrm{Gal}}}

\newcommand{\gln}{\ensuremath{\operatorname{GL}}}
\newcommand{\Gr}{\ensuremath{\mathrm{Gr}}}

\newcommand{\Hom}{\ensuremath{\mathrm{Hom}}}

\newcommand{\intal}{\ensuremath{\mathrm{intal}}}
\newcommand{\inte}{\ensuremath{\mathrm{int}}}

\newcommand{\Iw}{\ensuremath{\mathrm{Iw}}}

\newcommand{\modu}{\ensuremath{\mathrm{\;mod\;}}}

\newcommand{\ord}{\ensuremath{\mathrm{ord}}}

\newcommand{\pitilde}{\ensuremath{\widetilde{\pi}}}

\newcommand{\Qbar}{\ensuremath{\overline{Q}}}
\newcommand{\rig}{\ensuremath{\mathrm{rig}}}
\newcommand{\ring}{\ensuremath{\mathrm{ring}}}

\newcommand{\rec}{\ensuremath{\mathrm{rec}}}
\newcommand{\red}{\ensuremath{\mathrm{red}}}
\newcommand{\reg}{\ensuremath{\mathrm{reg}}}
\newcommand{\rep}{\ensuremath{\mathrm{rep}}}

\newcommand{\res}{\ensuremath{\mathrm{res}}}

\newcommand{\Spec}{\ensuremath{\mathrm{Spec}}}
\newcommand{\Sp}{\ensuremath{\mathrm{Sp}}}

\newcommand{\st}{\ensuremath{\mathrm{st}}}

\newcommand{\tr}{\ensuremath{\mathrm{tr}}}

\newcommand{\ur}{\ensuremath{\mathrm{ur}}}
\newcommand{\val}{\ensuremath{\mathrm{val}}}
\newcommand{\WBC}{\ensuremath{\mathrm{WBC}}}
\newcommand{\WD}{\ensuremath{\mathrm{WD}}}

\newcommand{\cf}{{\it cf.}}
\newcommand{\eg}{{\it eg}}

\newcommand{\ie}{{\it i.e.}}
\newcommand{\loccit}{{\it loc.\,cit}}

\newcommand{\mo}{{-1}}
\newcommand{\hra}{\hookrightarrow}
\newcommand{\hla}{\hookleftarrow}

\newcommand{\xra}{\xrightarrow}

\newcommand{\bbA}{\ensuremath{\mathbb{A}}}

\newcommand{\bbC}{\ensuremath{\mathbb{C}}}

\newcommand{\bbG}{\ensuremath{\mathbb{G}}}

\newcommand{\bbQ}{\ensuremath{\mathbb{Q}}}
\newcommand{\bbQbar}{\ensuremath{\overline\bbQ}}
\newcommand{\bbQp}{\ensuremath{{\mathbb{Q}_p}}}
\newcommand{\bbQpbar}{{\ensuremath{\overline{\mathbb{Q}}_p}}}
\newcommand{\bbQl}{\ensuremath{\mathbb{Q}_\ell}}
\newcommand{\bbQlbar}{\ensuremath{\overline\bbQ_\ell}}

\newcommand{\bbT}{\ensuremath{\mathbb{T}}}
\newcommand{\bbZ}{\ensuremath{\mathbb{Z}}}

\newcommand{\bbZp}{\ensuremath{{\mathbb{Z}_p}}}

\newcommand{\calD}{\ensuremath{\mathcal{D}}}

\newcommand{\calH}{\ensuremath{\mathcal{H}}}

\newcommand{\calK}{\ensuremath{\mathcal K}}
\newcommand{\calKbar}{\ensuremath{{\overline{\mathcal{ K}} }  }}
\newcommand{\calL}{\ensuremath{\mathcal L}}
\newcommand{\calLbar}{\ensuremath{{\overline{\mathcal{ L}} }  }}

\newcommand{\calO}{\ensuremath{\mathcal{O}}}

\newcommand{\calQ}{\ensuremath{\mathcal{Q}}}
\newcommand{\calQbar}{\ensuremath{{\overline{\mathcal{ Q}} }  }}
\newcommand{\calR}{\ensuremath{\mathcal{R}}}

\newcommand{\calT}{\ensuremath{\mathcal{T}}}

\newcommand{\calV}{\ensuremath{\mathcal{V}}}
\newcommand{\calW}{\ensuremath{\mathcal {W}}}

\newcommand{\calZ}{\ensuremath{\mathcal {Z}}}

\newcommand{\fraka}{\ensuremath{\mathfrak{a}}}

\newcommand{\frakl}{\ensuremath{\mathfrak{l}}}
\newcommand{\frakm}{\ensuremath{\mathfrak{m}}}

\newcommand{\frakp}{\ensuremath{\mathfrak{p}}}

\newcommand{\scrD}{\ensuremath{\mathscr{D}}}

\newcommand{\scrL}{\ensuremath{\mathscr{L}}}
\newcommand{\scrLbar}{\ensuremath{{\overline{\mathscr{ L}} }  }}

\newcommand{\scrO}{\ensuremath{\mathscr{O}}}

\begin{document}
\title{Purity for families of Galois representations}

\author{Jyoti Prakash Saha}

\address{D\'epartement de Math\'ematiques, B\^atiment 425, Facult\'e des Sciences d'Orsay, Universit\'e Paris-Sud 11, 91405 Orsay Cedex, France}

\email{jyoti-prakash.saha@math.u-psud.fr}




\begin{abstract}
We formulate a notion of purity for $p$-adic big Galois representations and pseudorepresentations of Weil groups of $\ell$-adic number fields for $\ell\neq p$. This is obtained by showing that all powers of the monodromy of any big Galois representation stay ``as large as possible'' under pure specializations. The role of purity for families in the study of the variation of local Euler factors, local automorphic types along irreducible components, the intersection points of irreducible components of $p$-adic families of automorphic Galois representations is illustrated using the examples of Hida families and eigenvarieties. Moreover, using purity for families, we improve a part of the local Langlands correspondence for $\gln_n$ in families formulated by Emerton and Helm. 
\end{abstract}

\maketitle

\tableofcontents
\section{Introduction}
\subsection{Motivation}
Let $r$ be a geometric Galois representation of the absolute Galois group 
of a number field with coefficients in $\bbQpbar$. Then the restriction $r_v$ of $r$ to a decomposition group at any given finite place $v$ not dividing $p$ is potentially unipotent by Grothendieck's monodromy theorem (see \cite[p.\,515--516]{SerreTate}). 
Given a projective smooth variety $X$ over a finite extension $K$ of $\bbQl$, 
the weight-monodromy conjecture (\cite[Conjecture 3.9]{IllusieMonodromieLocale}) 
says that for any prime $p\neq \ell$ and any integer $i\geq 0$, the $\Gal (\overline K/K)$-representation $H^i_\et(X _{\bbQlbar}, \bbQp)$ is pure of weight $i$, 
\ie, the $i$-th shift of the associated monodromy filtration coincides with the associated weight filtration 
(see definition \ref{Defn: pure representations}). 
When $r$ is irreducible, 
the representation $r_v$ is expected to be pure. 
The Galois representations attached to cuspidal automorphic representations 
(which are algebraic in the sense of \cite[Definition 1.8]{ClozelAnnArbor1}) 
by the Langlands correspondence (which is often conjectural) 
provide ample examples of geometric representations. 
The purity of the restrictions of $p$-adic automorphic Galois representations to decomposition groups 
at places outside $p$ is known in many cases due to works of 
Carayol 
\cite{Carayol86RepresentationAssocieHiblModForm}, 
Harris, Taylor 
\cite{HarrisTaylorGeometryCohomologyShimuraVar}, 
Blasius 
\cite{BlasiusRamanujanConj}, 
Taylor, Yoshida 
\cite{TaylorYoshidaCompatibilityLocalGlobalCorr}, 
Shin \cite{ShinGalReprCompactShimuraVar}, 
Caraiani \cite{CaraianiLocalGlobalCompatibility}, 
Scholze \cite{ScholzePerfectoidSpaces}, 
Clozel \cite{ClozelPurityReignsSupreme} {\it et.\,al}.
Following works of 
Hida \cite{HidaGalrepreord, HidaIwasawa, HidaControThmPNearlyOrdinaryCohGroupSLn}, Mazur \cite{MazurDeformingGaloisRepr}, 
Coleman, Mazur \cite{ColemanMazurEigencurve}, 
Chenevier \cite{ChenevierGLn}, 
Bella\"iche, Chenevier \cite{BellaicheChenevierU3}
{\it et.\,al.}, 
automorphic Galois representations are believed to live in $p$-adic families. 
Thus it is desirable to have a notion of purity for families. 
The goal of this article is to provide a formulation of this notion 
and to discuss its applications to $p$-adic families of Galois representations.

\subsection{Purity for families}
\label{SubSec: why called big purity}
The most naive way to formulate purity for big Galois representations would be to relate the 
monodromy filtration with the weight filtration. 
However the Frobenius eigenvalues on a big Galois representation 
are elements of a ring of large Krull dimension and are not algebraic numbers in general, precluding the 
possibility of considering the weight filtration. 
Thus a formulation of purity for big Galois representations is not straightforward. 
On the other hand, it is natural to expect that such a formulation 
should include a compatibility statement at pure specializations.

This formulation is achieved in theorem \ref{Thm: purity big}, which we call {\it purity for 
big Galois representations} because it says that the structures of Frobenius-semisimplifications 
of Weil-Deligne parametrizations of pure specializations of a ($p$-adic) big Galois representation (of the Weil group of an $\ell$-adic number field with $\ell\neq p$) are ``rigid''. 
In other words, it says that given a pure Weil-Deligne representation, its lifts to Weil-Deligne representations over integral domains have the ``same structure''. 

An important example of families of Galois representations comes from eigenvarieties. 
The traces of the Galois representations attached to the arithmetic points of an eigenvariety 
are interpolated by a pseudorepresentation defined over the global sections of the eigenvariety. 
Thus a notion of purity for pseudorepresentations is indispensable for the understanding of 
various local properties of the arithmetic points of eigenvarieties. 
This is provided by theorem \ref{Thm: purity for pseudorepresentations}, which we call 
{\it purity for pseudorepresentations}. 
It says that given an $\scrO$-valued pseudorepresentation $T$ of the 
Weil group of an $\ell$-adic number field 
(where $\scrO$ is a characteristic zero domain over $\bbZp$ with $p \neq \ell$), 
the Frobenius-semisimplification of two Weil-Deligne representations over two domains 
(containing $\scrO$ as a subalgebra) have the ``same structure'' 
if their traces are equal to $T$ and each of them has a pure specialization. 
This is deduced using purity for big Galois representations.

By \cite[Lemma 7.8.11]{BellaicheChenevierAsterisQUE}, 
around each nonempty admissible open affinoid subset $U$, the pseudorepresentation defined over the global sections of an 
eigenvariety lifts to a Galois representation on a finite type module over some integral 
extension of the normalization of $\calO(U)$. But this module is not known to be free over its coefficient ring. This 
forbids us from applying theorem \ref{Thm: purity for pseudorepresentations} to eigenvarieties. 
To circumvent this problem, we prove a general result in theorem \ref{Thm: purity sum} which we explain now. 
Let $T:G_F\to \scrO$ (where $\scrO$ is a characteristic zero domain over $\bbZp$ 
and $F$ is a number field) be a pseudorepresentation 
which is equal to 
$T_1+\cdots+T_n$ where $T_1:G_F\to \scrO, \cdots, T_n:G_F\to \scrO$ 
are traces of some irreducible $G_F$-representations over $\overline Q(\scrO)$ 
whose restrictions to the Weil group of a finite place $w\nmid p$ of $F$ are monodromic (see definition \ref{Defn: Monodromic}). 
Let $\calO, \calO'$ be two domains containing $\scrO$ as a subalgebra. 
Let $\frakp$ (resp. $\frakp'$) be a prime ideal of $\calO$ (resp. $\calO'$) 
such that the residue field $\kappa=\calO_\frakp/\frakp \calO_\frakp$ 
(resp. $\kappa'=\calO'_{\frakp'}/{\frakp'}\calO'_{\frakp'}$) is a finite extension of $\bbQp$ 
and the Henselization $\calO_\frakp^h$ (resp. $\calO_{\frakp'}'^h$) of $\calO_\frakp$ 
(resp. $\calO'_{\frakp'}$) is Hausdorff. 
Suppose $T_1\modu \frakp, \cdots, T_n\modu \frakp$ 
(resp. $T_1\modu \frakp', \cdots, T_n\modu \frakp'$) 
are traces of irreducible 
$G_F$-representations $\rho_1, \cdots, \rho_n$ (resp. $\rho_1', \cdots, \rho_n'$) 
over $\overline \kappa$ (resp. $\overline \kappa'$) and 
the representations 
$\rho_1|_{G_w}, \cdots, \rho_n|_{G_w}$ 
(resp. $\rho_1'|_{G_w}, \cdots, \rho_n'|_{G_w}$) are pure. 
Then using \cite[Th\'eor\`eme 1]{NyssenPseudoRepresentations} and purity for pseudorepresentations, 
we show in theorem \ref{Thm: purity sum} that the structure of $\WD(\oplus_{i=1}^n \rho_i|_{W_w})^\Frss$ and 
$\WD(\oplus_{i=1}^n \rho_i'|_{W_w})^\Frss$ are ``rigid''. 
Thus theorem \ref{Thm: purity sum} 
can be applied to eigenvarieties to prove the ``rigidity'' of the Frobenius-semisimplifications of the 
Weil-Deligne parametrizations of the local Galois representations attached to the arithmetic points 
that lie within the ``irreducibility and purity locus'' 
(see definition \ref{Defn: locus}) of 
(certain pseudorepresentations attached to) pseudorepresentations 
defined over global sections of eigenvarieties. 
Henceforth, by {\it purity for families}, we refer to theorem \ref{Thm: purity big}, 
\ref{Thm: purity for pseudorepresentations}, \ref{Thm: purity sum}. 

\subsection{Statement of purity for big Galois representations}
\label{SubSec: Set up}
In theorem \ref{Thm: introduction: purity big} below, 
we state a special case of theorem \ref{Thm: purity big}. 
We refer to 
\S 
\ref{Sec: Purity for pseudorepresentations} 
for the statements of theorem 
\ref{Thm: purity for pseudorepresentations}, 
\ref{Thm: purity sum}.

Let $p, \ell$ be two distinct primes and $K$ denote a finite 
extension of $\bbQl$. Denote the absolute Galois group of $K$ by $G_K$. 
Let $I_K$ denote the inertia group and $W_K$ denote the Weil group. 
Let $q$ denote the cardinality of the 
residue field $k$ of the ring of integers $\calO_K$ of $K$. 
Fix an element $\phi\in G_K=\Gal(\overline K/K)$ which lifts the geometric Frobenius $\Fr_k\in G_k=\Gal(\overline k/k)$. 
The Frobenius-semisimplification of the Weil-Deligne parametrization 
of a monodromic (see definition \ref{Defn: Monodromic}) 
representation $V$ of $W_K$ is denoted by $\WD(V)^\Frss$. 
We refer to \S 
\ref{SubSec: LLC introduction} and \S \ref{SubSec: notation} for few more notations. 
From now on by a {\it big Galois representation}, 
we mean a monodromic representation $\rho: W_K \to \Aut_\calR (\calT)$ of $W_K$ on 
a free $\calR$-module $\calT$ of finite rank where $\calR$ is a 
domain containing $\bbZp$ as a subalgebra. 
Note that if $\calR$ is a local ring with finite residue field and 
$\rho|_{I_K}$ is continuous, then $\rho$ is monodromic by Grothendieck's monodromy theorem 
(see \cite[p.~515--516]{SerreTate}). 
Also if $\calR$ is an affinoid algebra over $\bbQp$ 
and $\rho|_{I_K}$ is continuous, then $\rho$ is monodromic by 
Grothendieck's monodromy theorem 
(see \cite[Lemma 7.8.14]{BellaicheChenevierAsterisQUE}). 
Denote the $W_K$-representation $\calT \otimes_\calR \Qbar (\calR)$ by $\calV$ 
and let 
$$\WD(\calV)^\Frss \simeq \bigoplus_{i=1}^m \Sp_{t_i}(\chi_i\otimes \rho_i)_{/{\Qbar(\calR)}}$$
be the isomorphism of Weil-Deligne representations 
(as in equation \eqref{Eqn: Decomposition of WD of big V into Sp t r}) 
where 
$m, t_1\leq t_2\leq \cdots \leq t_m$ are positive integers, 
$\chi_1, \cdots, \chi_m$ are $(\calR^\intal)^\times$-valued unramified characters of $W_K$ and 
$\rho_1, \cdots, \rho_m$ are irreducible Frobenius-semisimple representations of $W_K$ over 
$\calR^\intal[1/p]$ with finite image. 
Given a field $E$ and a ring homomorphism $f:\calR\to E$, the $W_K$-representation $\calT\otimes_{\calR, f} E$ is denoted by $V_f$. 
We fix an isomorphism $\iota_p: \bbQpbar \simeq \bbC$ and let $\rec$ denote the reciprocity map as in \S \ref{Sec: Automorphic and Galois types}. 

\begin{mytheorem}[Purity for big Galois representations]
\label{Thm: introduction: purity big}
Let $\lambda: \calR \to \bbQpbar$ be a $\bbZp$-algebra homomorphism 
such that $V_\lambda$ is pure. 
Then the following hold. 
\begin{enumerate}
\item The rank of no power of the monodromy of $\calT_p$ decreases after specializing at $\lambda$. 
\item The Weil-Deligne representations 
$\WD(V_\lambda)^\Frss$ and $\oplus_{i=1}^m \Sp_{t_i}(\lambda^\intal \circ (\chi_i \otimes \rho_i))_{/\bbQpbar}$ 
are isomorphic. 
\item The polynomial $\Eul(\calV, X)^\mo$ has coefficients in $\calR^\inte$ and 
\begin{equation}
\label{Euler Interpolation}
\lambda(\Eul(\calV, X)^\mo) = \Eul(V_\lambda, X)^\mo.
\end{equation}
\item 
If $\xi:\calR \to \bbQpbar$ is a $\bbZp$-algebra homomorphism 
such that $V_\xi$ is pure, 
then the automorphic types of $\rec (\iota_p(\WD(V_\xi)^\Frss))$ and 
$\rec (\iota_p(\WD(V_\lambda)^\Frss))$ are the same. 
\end{enumerate}
Moreover, for any field extension $\calK$ of $\bbQp$ and any $\bbZp$-algebra homomorphism $\mu: \calR \to \calK$ with $\lambda (\ker \mu)=0$, 
the Weil-Deligne representation $\WD(V_\mu\otimes_\calK\calKbar)^\Frss$ is isomorphic to $\oplus_{i=1}^m \Sp_{t_i}(\mu^\intal \circ (\chi_i \otimes \rho_i))_{/\calKbar}$. 
\end{mytheorem}

Note that when $\scrD:=\{\ker \lambda~|~\lambda \in \Hom_{\bbZp\text{-}\alg}(\calR, \bbQpbar) ,V_\lambda \text{ is pure} \}$ is dense in 
$\Spec(\calR)$, 
using Hilbert's nullstellensatz, 
some of the above results (for example, equation \eqref{Euler Interpolation}) can 
be proved for $\lambda= \iota_\frakp \circ \textrm{mod}~ \frakp$ for $\frakp$ varying in a dense subset of $\scrD$ 
(here $\iota_\frakp$ denotes a $\bbZp$-algebra injection from $\calR/\frakp$ to $\bbQpbar$).

\subsection{Applications}
Theorem \ref{Thm: purity big}, \ref{Thm: purity for pseudorepresentations}, \ref{Thm: purity sum} turn out to be useful in the study of some arithmetic 
aspects of $p$-adic families of Galois representations. 
For example, 
the local Langlands correspondence for $\gln_n$ in families, 
the local automorphic types of arithmetic points of $p$-adic families, 
the geometry of the underlying spaces of families etc. 
These are studied in theorem \ref{Thm: LLC families}, 
\ref{Thm: Application: Hida cusp}, 
\ref{Thm: Application: Hida unitary}, 
\ref{Thm: Application: Eigen}. 
In this section, we state a special case of 
theorem \ref{Thm: LLC families} 
and explain the content of theorem 
\ref{Thm: Application: Hida cusp}, 
\ref{Thm: Application: Hida unitary}, 
\ref{Thm: Application: Eigen}.

\subsubsection{Local Langlands correspondence for $\gln_n$ in families}
\label{SubSec: LLC introduction}
The local Langlands correspondence, proved by 
Harris, Taylor 
\cite{HarrisTaylorGeometryCohomologyShimuraVar}, 
asserts that there is a canonical bijection between the isomorphism classes of $n$-dimensional Frobenius-semisimple complex 
Weil-Deligne representation of $W_K$ and the isomorphism classes of irreducible admissible representations of 
$\gln_n(K)$. This is extended to $p$-adic families of representations of $G_K$ by Emerton and Helm in \cite{EmertonHelmLLC}. 
We state a special case of it. 

First, we fix some notations. 
Given a monodromic representation $\rho:G_K \to \gln_n(L)$ over a field $L$ of characteristic zero, 
the representation attached to $\WD(\rho)^\Frss$ by the extension (\cite[\S 4.2]{EmertonHelmLLC}) of the modified local Langlands correspondence of 
Breuil and Schneider (see \cite[p.~161--164]{BreuilSchneiderFirstStep}) is denoted by $\pi(\rho)$ 
and the smooth contragredient of $\pi(\rho)$ is denoted by $\widetilde \pi(\rho)$ 
(the representations $\pi(\rho), \widetilde \pi(\rho)$ are equal to 
$\pi(\WD(\rho)^\Frss), \widetilde \pi(\WD(\rho)^\Frss)$ respectively in the notation of \cite{EmertonHelmLLC}). 
Let $A$ be a complete reduced $p$-torsion free Noetherian local ring with finite residue field of characteristic $p$. 
The residue field of a prime ideal $\frakp$ of $A$ is denoted by $\kappa(\frakp)$. 

Given a continuous Galois representation $r: G_K \to \gln_n(A)$, there exists at most one admissible smooth 
$\gln_n(K)$-representation $V$ over $A$, up to isomorphism, satisfying some conditions (conditions (1), (2), (3) of 
\cite[Theorem 1.2.1]{EmertonHelmLLC}). 
Suppose such a $V$ exists. 
Let $\calD$ denote the set of primes $\frakp$ of $A[1/p]$ such that the number of irreducible components of $\Spec A[1/p]$ 
passing through $\frakp$ is one. 
By \cite[Theorem 6.2.5]{EmertonHelmLLC}, for any $\frakp \in \calD$, there is a $\gln_n(K)$-equivariant surjection 
\begin{equation}
\label{Eqn: Surjection Introduction}
\pitilde (\kappa(\frakp) \otimes_A r) \to \kappa(\frakp) \otimes_AV.
\end{equation}
Let $\calD'$ denote the set of primes $\frakp$ in $\calD$ for which the above map is an isomorphism. 
Then $\calD'$ contains an open dense subset $U'$ of $\Spec A[1/p]$ by \cite[Theorem 1.2.1]{EmertonHelmLLC}. 
By theorem \ref{Thm: intro: LLC} below, 
$\calD'$ contains all the elements of $\calD$ 
that are contained in kernel of pure specializations. 
For a more general result, we refer to 
theorem \ref{Thm: LLC families} which is proved using 
\cite[Theorem 6.2.1, 6.2.5, 6.2.6]{EmertonHelmLLC} and theorem \ref{Thm: purity big}. 

\begin{mytheorem}
\label{Thm: intro: LLC}
Let $V, \calD, \calD'$ be as above. Suppose $V$ exists. 
Let $\frakp$ be a prime in $\calD$. Suppose that there exists a $\bbZp$-algebra homomorphism 
$i_\frakp : A/\frakp \to \bbQpbar$ such that 
$r\otimes_A A/\frakp \otimes _{A/\frakp, i_\frakp}\bbQpbar$ is pure. Then $\frakp$ lies in $\calD'$. 
\end{mytheorem}

Hida's theory 
of ordinary automorphic representations provide continuous representations of 
absolute Galois group of number fields with coefficients in rings of the form 
$A$. 
So their restriction to decomposition groups at places not dividing $p$ gives 
representations of the form $r$, to which \cite[Theorem 6.2.1, 6.2.5, 6.2.6]{EmertonHelmLLC} and 
theorem \ref{Thm: LLC families} apply. 
On the other hand, overconvergent forms also form families, although of rather different 
nature, for instance, there are examples of such families whose coefficient rings are not local 
(and there are also families of overconvergent forms defined over local rings, see \cite{AISOverconv}). 
The local Langlands correspondence is not yet extended to families defined over non-local rings or to the case when $A$ is 
an affinoid algebra. 
However, the coefficient rings $\calR, \scrO, \calO, \calO'$ as in 
theorem \ref{Thm: purity big}, 
\ref{Thm: purity for pseudorepresentations}, 
\ref{Thm: purity sum} 
are quite general, for instance, $\calR, \scrO$ are not assumed to be local or Noetherian. 
So once a notion of local Langlands correspondence for more general families 
is established, it is likely that one could use theorem \ref{Thm: purity big} 
\ref{Thm: purity for pseudorepresentations}, \ref{Thm: purity sum} 
to show that the extension (as in \cite[\S 4.2]{EmertonHelmLLC}) of the Breuil-Schneider modified local Langlands correspondence is 
interpolated at all the primes contained in the kernel of pure specializations.

\subsubsection{Hida families and eigenvarieties}

Given a $p$-adic family of Galois representations of the absolute Galois group 
of a number field, the variation of the Frobenius-semisimplifications 
of the Weil-Deligne parametrizations of the local Galois representations 
attached to the members 
at places outside $p$ can be studied using theorem 
\ref{Thm: purity big}, \ref{Thm: purity for pseudorepresentations}, 
\ref{Thm: purity sum}. 
Thus purity for families illustrates the 
variation of local Euler factors of the 
arithmetic points of $p$-adic families of automorphic Galois representations 
and also the variation of local automorphic types of arithmetic points when local-global compatibility is known. 
In \S \ref{Sec: Constancy of automorphic types}, 
we explain this variation using the examples of Hida family of cusp forms, 
Hida family of ordinary automorphic representations of definite unitary groups, 
eigenvariety for definite unitary groups. 
We refer to theorem 
\ref{Thm: Application: Hida cusp}, 
\ref{Thm: Application: Hida unitary}, 
\ref{Thm: Application: Eigen} 
for the precise statements. 
Roughly speaking, these three results state that the ``Galois types'' of the local Galois representations 
attached to the arithmetic points of any given irreducible component of these families are constant (under some hypotheses). 
In the proofs of theorem 
\ref{Thm: Application: Hida cusp}, 
\ref{Thm: Application: Hida unitary}, 
\ref{Thm: Application: Eigen}, 
we do not use the fact that the arithmetic points of these families form a dense subset. 
Moreover in theorem \ref{Thm: Application: Hida cusp}, 
we do not assume that the residual representation attached to (a branch of) the Hida 
family of ordinary cusp forms is residually absolutely irreducible. 
However in theorem \ref{Thm: Application: Hida unitary}, 
we only consider those branches of the Hida family (of ordinary 
automorphic representations of a definite unitary group) 
whose associated minimal primes are contained in non-Eisenstein maximal ideals. 
In theorem \ref{Thm: Application: Eigen}, 
we assume that each irreducible component of the eigenvariety attached to 
the definite unitary group $U(m)$ contains at least one arithmetic point such that 
its associated automorphic representation $\pi$ is regular at infinity, 
the semisimple conjugacy class of $\pi_p$ has $m$ distinct eigenvalues and 
the weak base change of $\pi$ is cuspidal. 
We assume further that the Galois representation attached to an 
automorphic representation $\pi$ of $U(m)$ of regular weight at infinity 
is irreducible if the weak base change of $\pi$ is cuspidal. 
For related results, we refer to 
\cite[\S 12.7.14]{NekovarSelmerComplexes}, 
\cite[\S 7.5.3, 7.8.4]{BellaicheChenevierAsterisQUE}, 
\cite[Theorem A]{PaulinLocalGlobal}.

\subsection{Organization}
\label{SubSec: Organization}

The main results obtained in this article are theorem 
\ref{Thm: purity big}, 
\ref{Thm: purity for pseudorepresentations}, 
\ref{Thm: purity sum}, 
\ref{Thm: LLC families}, 
\ref{Thm: Application: Hida cusp}, 
\ref{Thm: Application: Hida unitary}, 
\ref{Thm: Application: Eigen}.

In \S \ref{Sec: local Galois repr l}, 
we introduce the notion of 
Weil-Deligne representations over domains following 
\cite[8.4--8.6]{DeligneConstantesDesEquationsFunctional}, 
\cite[p.\,77--78]{TaylorGaloisReprExtended}. 
Then we recall the notion of pure representations and Euler factors. 
We begin section \ref{Sec: Automorphic and Galois types} by 
fixing a unitary local Langlands correspondence. 
Then we introduce the modified local Langlands correspondence 
of Breuil-Schneider (\cite[p.~161--164]{BreuilSchneiderFirstStep}) 
and the extension of this modification due to Emerton and Helm 
(\cite[\S 4.2]{EmertonHelmLLC}). 
In \S \ref{SubSec: LLC families notations}, we recall the 
formulation of the local Langlands correspondence for $\gln_n$ in $p$-adic families by Emerton and Helm. 
The notion of automorphic type is defined in \S \ref{SubSec: Automorphic Galois type}.

In the next section, we prove theorem \ref{Thm: purity big}. 
In its proof, we crucially use 
(through equation \eqref{Eqn: Decomposition of WD of big V into Sp t r} for instance) the hypothesis 
that the ring $\calR$ is a domain. We cannot expect to prove theorem 
\ref{Thm: purity big} when 
the ring $\calR$ is replaced by a more general ring, an example being a ring 
with finitely many minimal primes. 
In fact a crucial step in its proof 
is to express the trace of $\calV$ as a sum of traces of irreducible Frobenius-semisimple representations 
over $\calR^\intal$ and then to pin down the factors of powers of the character $|\Art_K^\mo|_K$ in them. 
The amount of these factors is governed by the size of the Jordan blocks of the monodromy 
of $\calV$. 
When the coefficient ring $\calR$ of $\calT$ is not 
a domain, then the shapes of the Jordan blocks of the images of its monodromy 
in the stalks of $\Spec(\calR)$ at the generic points need not be independent 
of the generic points. Thereby, in no reasonable manner, it is possible to pin down the factors of 
powers of $|\Art_K^\mo|_K$ present in the representations stated above. Even in the 
very simple case where $\calR= \bbQp[[X]]\times \bbQp[[X]] \times \bbQp[[X]]$, $\calV$ is semistable 
and $N\in M_3(\calR)$ is the strictly upper triangular matrix with 
$N_{12} = (X, 0, 0), N_{13}= 0, N_{23} = (0, X-1, 0)$, 
we cannot track the `right' factors of powers of $q$ in the 
characteristic roots of $\phi$ on $\calV$. Thus it seems hard to have a reasonable analogue of 
equation \eqref{Eqn: Decomposition of WD of big V into Sp t r} that could 
lead to an analogue of theorem \ref{Thm: purity big} when $\calR$ is a more general 
ring than a domain. So we are compelled to assume that $\calR$ is a domain.

In \S \ref{Sec: Purity for pseudorepresentations}, we use purity for big Galois representations 
to prove purity for pseudorepresentations (theorem \ref{Thm: purity for pseudorepresentations}) 
by an induction argument. 
Theorem \ref{Thm: purity sum} follows as a corollary of theorem \ref{Thm: purity for pseudorepresentations}. 
Section \ref{Sec: LLC families} 
uses results from \cite{EmertonHelmLLC} and theorem \ref{Thm: purity big} 
to prove theorem \ref{Thm: LLC families} about local Langlands correspondence 
for $\gln_n$ in families. 
In section \ref{Sec: Constancy of automorphic types}, 
using the examples of Hida family of ordinary cusp forms, 
Hida family of ordinary automorphic representations of definite unitary groups 
and eigenvarieties, we illustrate the role of theorem 
\ref{Thm: purity big}, 
\ref{Thm: purity for pseudorepresentations}, 
\ref{Thm: purity sum}
in the study of the local data (\eg. local Euler factors, local automorphic types, 
Weil-Deligne representations) associated to the members of 
$p$-adic families of automorphic Galois representations 
(theorem \ref{Thm: Application: Hida cusp}, 
\ref{Thm: Application: Hida unitary}, 
\ref{Thm: Application: Eigen}).

\addtocontents{toc}{\protect\setcounter{tocdepth}{-1}}

\subsection{Notations}
\label{SubSec: notation}
For every field $F$, we fix an algebraic closure $\overline F$ of it. For any finite place $v$ of a number field $E$, the decomposition group 
$\Gal(\overline E_v/E_v)$ is denoted by $G_v$. Let $W_v\subset G_v$ (resp. $I_v\subset G_v$) denote the Weil group (resp. inertia group) 
and $\Fr_v\in G_v/I_v$ denote the geometric Frobenius element. 
We fix embeddings $\bbC \mathop{\hla}\limits^{i_\infty} \bbQbar \mathop{\hra}\limits^{i_p} \bbQpbar$ once and for all. 
The largest reduced quotient of a ring $A$ is denoted by $A_\red$ and 
the map $A_\red \to B_\red$ induced by a ring homomorphism $f:A\to B$ is denoted by 
$f_\red$. 
The fraction field of a domain $A$ is denoted by $Q(A)$ and the field 
$\overline{Q(A)}$ is denoted by $\overline Q (A)$. 
If $R$ is a ring with a unique minimal prime ideal, then 
the integral closure of $R_\red$ in $Q(R_\red)$ (resp. $\Qbar (R_\red)$) is denoted by $R^\inte$ (resp. $R^\intal$). 
If $f:R\to S$ is a ring homomorphism where $S$ is a ring with a unique minimal prime ideal, 
then the map $f_\red$ has an extension to a map $R^\intal\to S^\intal$. We fix one such map 
and denote it by $f^\intal$. 
If an integer $m$ is nonzero in $S^\red$, then the unique extension 
$R^\intal[1/m] \to S^\intal[1/m]$ of $f^\intal$ is also denoted by $f^\intal$. 
By a representation of a group $G$ on a module $M$ over a ring $A$, 
we mean a group homomorphism $G\to \Aut_A(M)$ (even if $G$ is 
a topological group) unless otherwise 
stated.

\addtocontents{toc}{\protect\setcounter{tocdepth}{3}}

\section{Local Galois representations}\label{Sec: local Galois repr l} 

Let $\varpi$ 
denote a uniformizer of $\calO_K$ and $\val_K:K^\times\twoheadrightarrow \bbZ$ 
be the $\varpi$-adic valuation. Let $\mid\cdot\mid_K:=(\# 
k)^{-\val_K(\,\cdot\,)}$ be the corresponding norm. 
The Weil 
group $W_K$ is defined as the subgroup of $G_K$ consisting of elements which 
map to an integral power of $\Fr_k$ in $G_k$. 
The Artin map $\Art_K:K^\times \xra{\sim}W_K^\ab$ 
is normalized so that the uniformizing parameters go to the lifts of the geometric Frobenius 
element. Let $P_K\subset I_K$ 
denote the wild inertia subgroup.
Then given a compatible system $\zeta=(\zeta_n)_{\ell\nmid n}$ of primitive 
roots of unity, we have an isomorphism 
$t_\zeta:I_K/P_K\xra{\sim} \prod_{p\neq \ell}\bbZp$
such that $\sigma(\varpi^{1/n})=\zeta_n^{(t_\zeta(\sigma)\modu n)}\varpi^{1/n}$ for all $\sigma\in I_K/P_K$. 
By \cite[Theorem 7.5.2]{NSW}, 
for all $\sigma\in W_K$ and $\tau\in I_K$, 
we have 
$t_\zeta(\sigma\tau\sigma^\mo)=\varepsilon(\sigma)t_\zeta(\tau)$ 
where 
$\varepsilon:=\prod_{p\neq \ell}\varepsilon_p:G_K\to \prod_{p\neq 
\ell}\bbZ_p^\times$ 
is the product of the cyclotomic characters. For a prime $p\neq \ell$, let 
$t_{\zeta,p}:I_K \to \bbZp$ denote the composition of the projection $I_K \to I_K/P_K$, 
the map $t_\zeta$ and the 
projection from $\prod_{p\neq \ell} \bbZp$ to $\bbZp$. Define $v_K:W_K\to \bbZ$ by 
$\sigma|_{K^\ur}=\Fr_k^{v_K(\sigma)}$ for all $\sigma\in W_K$.

\begin{mydefinition}[{\cite[8.4.1]{DeligneConstantesDesEquationsFunctional}, 
\cite[p.\,77--78]{TaylorGaloisReprExtended}}]
\label{Defn: Weil-Deligne representations}
Let $A$ be a commutative domain of characteristic zero. 
\begin{enumerate}
\item A \textnormal{Weil-Deligne representation} of $W_K$ on a free $A$-module 
$M$ of finite rank is a triple $(r,M,N)$ consisting of a representation 
$r:W_K\to \Aut_A (M)$ and a nilpotent endomorphism $N\in \End_A(M)$ such that 
$r(I_K)$ is finite and for all 
$\sigma\in W_K$, 
$$r(\sigma)N r(\sigma)^{-1}=q^{-v_K(\sigma)}N $$
in $\End_{A[1/\ell]}(M\otimes_A A[1/\ell])$. 
The operator $N$ is called the 
\textnormal{monodromy} of $(r, M, N)$. 
\item 
A representation $\rho$ of $W_K$ on a free module $M$ of finite rank over a domain $A$ is said to be 
\textnormal{irreducible Frobenius-semisimple} 
if $M\otimes \Qbar(A)$ is irreducible, the action of $\phi$ on $M\otimes \Qbar(A)$ is semisimple and 
$\#\rho(I_K)<\infty$. 
\end{enumerate}
\end{mydefinition}

The sum of Weil-Deligne representations are defined in the 
usual way (see \cite[\S 31.2]{BushnellHenniart} for instance).

\begin{mydefinition}\label{Defn: Monodromic}
Let $A$ be a $\bbZp$-algebra of characteristic zero. Suppose $M$ 
be an $A$-module together with a $W_K$-action $\rho:W_K\to \Aut_A(M)$ on it. 
We say $M$ is \textnormal{monodromic} with \textnormal{monodromy} $N$ 
\textnormal{over} $K'$ if there exists a finite extension $K'/K$ and 
$N$ is a nilpotent element of $\End_{A[1/p]}(M\otimes_AA[1/p])$ such that for all 
$\tau\in I_{K'}$
$$\rho(\tau)=\exp(t_{\zeta, p}(\tau)N)$$
in $\End_{A[1/p]}(M\otimes_AA[1/p]).$
An $A$-module $M'$ equipped with an action of $G_K$ is said to be \textnormal{monodromic} 
if $M'|_{W_K}$ is monodromic. 
\end{mydefinition}

Suppose $(r,N)=(r,V,N)$ is a 
Weil-Deligne representation with coefficients in a field $L$ of characteristic 
zero which contains all the characteristic roots of all the elements of 
$r(W_K)$. Let $r(\phi)=r(\phi)^{ss}u=ur(\phi)^{ss}$ 
be the Jordan decomposition of $r(\phi)$ as the product of a diagonalizable 
matrix $r(\phi)^{ss}$ and a unipotent matrix $u$. Following 
\cite[8.5]{DeligneConstantesDesEquationsFunctional}, 
\cite[p.\,78]{TaylorGaloisReprExtended}, define 
$\tilde r(\sigma)=r(\sigma) u^{-v_K(\sigma)}$ 
for all $\sigma\in W_K$.
Then $(\tilde r, V, N)$ is a Weil-Deligne representation 
(by 
\cite[8.5]{DeligneConstantesDesEquationsFunctional} 
for example) 
and is called the {\it Frobenius semisimplification} of $(r,V, N)$ 
(\cf\,\cite[8.6]{DeligneConstantesDesEquationsFunctional}). 
It is denoted by $V^\Frss$. We say $(r, V, N)$ is {\it Frobenius-semisimple} if $\tilde 
r=r$.

\begin{mydefinition}
For an integer $t\geq 1$, a characteristic zero commutative domain $A$ with $\ell\in A^\times$ 
and a representation $(r, M)$ of $W_K$ on a free module $M$ of finite rank over $A$ with $\#r(I_K)<\infty$, 
we denote by $\Sp_t(r)_{/A}$ the Weil-Deligne representation with underlying 
module $M^t$ on which $W_K$ acts via 
$$r|\Art_K^\mo|_K ^{t-1}\oplus r|\Art_K^\mo|_K ^{t-2} \oplus \cdots \oplus 
r|\Art_K^\mo|_K \oplus r$$
and the monodromy $N$ induces an isomorphism from $r|\Art_K^\mo|_K^{i}$ to 
$r|\Art_K^\mo|_K^{i+1}$ for all $0\leq i\leq t-2$ and is zero on 
$r|\Art_K^\mo|_K ^{t-1}$. 
\end{mydefinition}

Let $\Omega$ denote an algebraically closed field of characteristic zero.

\begin{mydefinition}
A Weil-Deligne representation over $\Omega$ is said to be 
\textnormal{indecomposable} if it is not isomorphic to a direct sum of two 
nonzero Weil-Deligne representations over $\Omega$. 
\end{mydefinition}

\begin{mytheorem}\label{Thm: Structure of Frob ss Weil-Deligne representations}
Let $(\rho, V, N)$ be a Frobenius-semisimple Weil-Deligne representation over $\Omega$. 
Then it is isomorphic to 
$$\bigoplus_{i\in I}\Sp_{t_i}(r_i)_{/\Omega}$$
for some irreducible Frobenius-semisimple representations $r_i:W_K\to 
\gln_{n_i}(\Omega)$ and positive integers $t_i$. This decomposition is unique up to reordering and 
replacing factors by isomorphic factors.
\end{mytheorem}

\begin{proof}
This follows from the proof of 
\cite[Proposition 3.1.3 (i)]{DeligneFormesModulairesGL2}. 
\end{proof}

\begin{mydefinition}
\label{Defn: size}
Let $(\rho, V, N)$ be as above. 
Then the integer $\max \{t_i|i\in I\}$ is called the \textnormal{size} of $\rho$. 
\end{mydefinition}

\begin{mydefinition}
An \textnormal{indecomposable summand} of a Frobenius-semisimple Weil-Deligne 
representation $V$ over $\Omega$ is a Weil-Deligne subrepresentation of $V$ 
isomorphic to a summand $\Sp_{t_i}(r_i)_{/\Omega}$ via an isomorphism
$V\simeq \oplus_{i\in I}\Sp_{t_i}(r_i)_{/\Omega}$ 
as in theorem \ref{Thm: Structure of Frob ss Weil-Deligne representations}. 
\end{mydefinition}

While dealing with indecomposable summands of $V$, 
we always implicitly fix an isomorphism 
$V\simeq \oplus_{i\in I}\Sp_{t_i}(r_i)_{/\Omega}$ 
as in theorem \ref{Thm: Structure of Frob ss Weil-Deligne representations}.

\begin{myproposition}
\label{Prop: rationality in general}
Let $(r, N)$ be a Weil-Deligne representation over an integral domain $A$ of characteristic zero. 
Let $\bbQ^\cl$ denote the algebraic closure of $\bbQ$ in $\overline Q(A)$. 
Let $B$ be a subring of $A$ such that the characteristic polynomial of 
$r(g)$ has coefficients in $B$ for all $g\in W_K$. 
Then there exist
\begin{enumerate}[(i)]
\item positive integers $m, t_1 \leq  \cdots \leq t_m$,
\item $(B^\intal)^\times$-valued unramified characters $\chi_1, \cdots, \chi_m$ of $W_K$, 
\item irreducible Frobenius-semisimple representations $\rho_1, \cdots, \rho_m$ of $W_K$ 
with coefficients in $\bbQ^\cl$ with finite image  
\end{enumerate}
such that
$((r, N) \otimes_A \overline Q(A))^\Frss$ is isomorphic to 
$\oplus_{i=1}^m \Sp_{t_i}(\chi_i \otimes \rho_i)$. 
\end{myproposition}

\begin{proof}
By theorem \ref{Thm: Structure of Frob ss Weil-Deligne representations}, there exist 
positive integers $m, t_1\leq t_2\leq \cdots \leq t_m$, 
irreducible Frobenius-semisimple representations $r_1, \cdots, 
r_m$ of $W_K$ over $\overline Q(A)$ 
such that 
$((r, N) \otimes_A \overline Q(A))^\Frss$ is isomorphic to 
$\oplus_{i=1}^m \Sp_{t_i}(r_i)$. 
From the proof of \cite[28.6 Proposition]{BushnellHenniart}, 
it follows that for each $1\leq i\leq m$, there exists an unramified character 
$\chi_i:W_K\to \overline Q(A)^\times$ such that the $W_K$-representation 
$\chi_i^\mo \otimes r_i$ has finite image. So 
there exists an irreducible Frobenius-semisimple 
representation $\rho_i:W_K\to \gln_{d_i}(\bbQ^\cl)$ with finite image such that
$\chi_i^\mo \otimes r_i$ and $\rho_i$ are isomorphic over $\overline Q(A)$ 
(by \cite[Theorem 1]{TaylorGaloisReprAssociatedToSiegelModForms} for instance). 
So the product of $\chi_i(\phi)$ and a root of unity belongs to $B^\intal$. 
Thus $\chi_i(\phi)$ belongs to $B^\intal$ and similarly, $\chi_i(\phi)^\mo$ belongs to 
$B^\intal$. Hence $\chi_i$ has values in $B^\intal$. 
This proves the result. 
\end{proof}

\begin{mylemma}
\label{Lemma: irreducible Frobenius semisimple under specializations}
Let $r:W_K \to \gln_n(A)$ be an irreducible Frobenius-semisimple representation of $W_K$ 
with coefficients in a domain $A$ of characteristic zero. 
If $B$ is a domain and 
$f:A \to B$ is a ring homomorphism, then $f\circ r$ is also an irreducible 
Frobenius-semisimple representation. 
\end{mylemma}

\begin{proof}
Let $\bbQ^\cl$ denote the algebraic closure of $\bbQ$ in $\overline Q(A)$. 
By proposition \ref{Prop: rationality in general}, 
there exist an unramified character $\chi:W_K\to (A^\intal)^\times$ 
and an irreducible Frobenius-semisimple representation 
$\rho:W_K \to \gln_n(\bbQ^\cl)$ with finite image 
such that $r$ is isomorphic to $\chi\otimes \rho$ over $\overline Q(A)$. 
As $\rho(W_K)$ is finite, it is contained in $\gln_n(A^\intal[1/m])$ 
for some positive integer $m$. So 
$f^\intal(\rho)$ is isomorphic to $f^\intal (\chi^\mo\otimes r)
=f^\intal (\chi^\mo)\otimes f^\intal (r)
=f^\intal (\chi^\mo)\otimes f (r)$. Thus $f (r)$ is isomorphic to 
$f^\intal(\chi)\otimes f^\intal(\rho)$. This proves the lemma. 
\end{proof}

\begin{mydefinition}
\label{Defn: pure representations}
(\cf\,\cite[p.\,1014]{SchollHypersurfacesWeilConj}) A 
Frobenius-semisimple Weil-Deligne representation $V$ of $W_K$ over $\bbQpbar$ 
is said to be \textnormal{pure of weight $w\in \bbZ$} if the eigenvalues of one (and 
hence any) lift of the geometric Frobenius element on $\Gr_i M_\bullet$ are 
$q$-Weil numbers of weight $w+i$ where $M_\bullet$ denotes the monodromy 
filtration on $V$. 

A finite dimensional representation $V$ of $G_K$ or of $W_K$ over $\bbQpbar$ 
is said to be \textnormal{pure of 
weight} $w\in \bbZ$ if $V|_{W_K}$ is monodromic and the Frobenius semisimplification of the Weil-Deligne 
parametrization of $V|_{W_K}$ with respect to one (and hence any) choice of $\phi$ and 
$\zeta$ is pure of weight $w$.
\end{mydefinition}

We refer to \cite[Definition 2.5]{MilneMotivesOverFiniteFields} for the notion of Weil numbers 
and to \cite[equation 1.5.5]{IllusieMonodromieLocale} for the notion of monodromy filtration. 
\begin{myremark}
\label{Remark: Central summands determines purity}
\normalfont
Let $r_1, \cdots, r_m$ be irreducible Frobenius-semisimple representations of 
$W_K$ over $\bbQpbar$. Then by \cite[I.6.7, p. 166]{DeligneWeil2}, it follows that 
the Weil-Deligne representation $\oplus_{i=1}^m \Sp_{t_i}(r_i)_{/\bbQpbar}$ is pure 
of weight $w$ 
if and only if the $\phi$-eigenvalues on $r_1|\Art_K^\mo|_K^{(t_1-1)/2}$, $\cdots$, $r_m|\Art_K^\mo|_K^{(t_m-1)/2}$ are 
$q$-Weil numbers of weight $w$ (for any choice of a square root of $q$ in $\bbQpbar$). 
\end{myremark}

Let $\Omega$ be an algebraically closed field of characteristic zero. 
For a Weil-Deligne representation $(r,V, N)$ of $W_K$ over $\Omega$, 
its {\it Euler factor} $\Eul((r, N), X)$ is defined as 
the element $\det (1-X \phi|_{V^{I_K, N=0}})^\mo$ of $\Omega(X)$ 
where $V^{I_K, N=0}$ denotes the subspace of $V$ on which $I_K$ acts trivially 
and $N$ is zero. 
For a representation $\rho:\Gal(\overline E/E)\to \gln(V)$ of the 
absolute Galois group of a number field $E$ on a finite dimensional vector 
space $V$ over $\Omega$, 
its local {\it Euler factor} $\Eul_v(\rho, X)$ 
at a finite place $v$ of $E$ not dividing $p$ is defined 
to be the element $\Eul(\WD(V|_{W_v}), X)$ in $\Omega(X)$ 
if $V|_{W_v}$ is monodromic. 
We refer to \cite[p.\,85]{TaylorGaloisReprExtended} for details.

\section{The local Langlands correspondence and its extensions}
\label{Sec: Automorphic and Galois types}

The local Langlands correspondence for $\gln_n(K)$ is known due to 
works of Harris, Taylor \cite{HarrisTaylorGeometryCohomologyShimuraVar}. 
Depending on the required normalization, 
there are various choices of this correspondence. 
We prefer to work with the unitary local Langlands correspondence, 
which depends on the choice of a square root of $q$ in $\bbQpbar$, which we fix from now on. 
We denote the reciprocity map by $\rec$. 

\subsection{The modified local Langlands correspondence of Breuil-Schneider}
\label{SubSec: modified LLC}

We recall the modified local Langlands correspondence of Breuil-Schneider 
and its extension to Weil-Deligne representations with coefficients in 
any extension of $\bbQp$. We refer to \cite[p.~161--164]{BreuilSchneiderFirstStep} and 
\cite[\S 4.2]{EmertonHelmLLC} for details. 

Let $(\rho, N)$ be a Frobenius-semisimple Weil-Deligne representation of $W_K$ over $\bbQpbar$. 
Let $\pi(\rho, N)$ denote the indecomposable admissible representation of $\gln_n(K)$ 
over $K$ attached to $(\rho, N)$ via the Breuil-Schneider modified local Langlands correspondence 
(see \cite[p.~161--164]{BreuilSchneiderFirstStep}). 
To define the representation $\pi(\rho, N)$, one needs to choose a square root of $q$. However 
the representation $\pi(\rho, N)$ is independent of this choice. 

In \cite{EmertonHelmLLC}, this modified correspondence is extended to 
Frobenius-semisimple Weil-Deligne representations over an arbitrary field extension of $\bbQp$. 
For a Frobenius-semisimple Weil-Deligne representation $(\rho, N)$ of $W_K$ 
over an extension $L$ of $\bbQp$, let $\pi(\rho, N)$ denote the indecomposable admissible 
representation of $\gln_n(K)$ over $L$ attached to $(\rho, N)$ (see \S 4.2 of \loccit.).
The smooth contragredient of $\pi(\rho, N)$ is denoted by $\pitilde(\rho, N)$. 
If $r$ is a monodromic representation of $W_K$ on a finite dimensional vector space over $L$, 
then we denote by $\widetilde \pi (r)$ the representation $\widetilde \pi (\WD(r)^\Frss)$.

\subsection{The local Langlands correspondence for $\gln_n$ in families}
\label{SubSec: LLC families notations} 
Let $A$ be a complete reduced $p$-torsion free Noetherian local ring with finite residue field of characteristic $p$. 
Let $\frakm$ denote the maximal ideal of $A$. 
The residue field of a prime ideal $\frakp$ of $A$ is denoted by $\kappa(\frakp)$. 
For a prime ideal $\frakp$ of $A$, 
the mod $\frakp$ reduction of a representation $\rho$ of a group on an $A$-module is denoted by $\rho_\frakp$. 
We refer to \cite{EmertonHelmLLC} for unfamiliar notations and terminologies used below. 

\begin{mytheorem}
\label{Thm: EH 621}
Let $E$ be a number field and $S$ denote a finite set of non-archimedean places of $E$, none of which lie over $p$. 
For each $v\in S$, let $r_v: G_{E_v} \to \gln_n(A)$ be a continuous representation. 
Write $G= \prod_{v\in S} \gln_n(E_v)$. Then there exists at most one (up to isomorphism) 
admissible smooth representation $V$ of $G$ over $A$ satisfying the conditions below. 
\begin{enumerate}
\item $V$ is $A$-torsion free, \ie, all associated primes of $V$ are minimal primes of $A$.
\item For each minimal prime $\fraka$ of $A$, there is a $G$-equivariant isomorphism 
$$\bigotimes_{v\in S} \pitilde (r_{v, \fraka}) \xra{\sim} \kappa(\fraka) \otimes_A V.$$
\item The $G$-cosocle $\cosoc (V/\frakm V)$ of $V/\frakm V$ is absolutely irreducible and generic, while the kernel 
of the natural surjection $V/\frakm V\to \cosoc (V/\frakm V)$ contains no generic subrepresentations. 
\end{enumerate}

\end{mytheorem}

\begin{proof}
It is a part of \cite[Theorem 6.2.1]{EmertonHelmLLC}.
\end{proof}

When $V$ exists, we denote it by $\pitilde (\{r_v\}_{v\in S})$. 
When $S$ contains only one place, we denote $V$ by $\pitilde(r_v)$. 
By \cite[Proposition 6.24]{EmertonHelmLLC}, the $A[G]$-module $\pitilde(\{r_v\}_{v\in S})$ exists 
if and only if each of the individual $A[\gln_n(E_v)]$-modules $\pitilde(r_v)$ exists. 
For a minimal prime $\fraka$ of $A[1/p]$, the monodromy of $r_{v, \fraka}$ 
is denoted by $N_v(\fraka)$.

\begin{mytheorem}
\label{Thm: EH 626}
Let $S$ be as in theorem \ref{Thm: EH 621} and $\frakp$ be a prime ideal of $A[1/p]$. Let $\fraka_1, \cdots, \fraka_s$ be the minimal primes of 
$A$ contained in $\frakp$. For each $i=1, \cdots, s$, let $V_i$ 
be the maximal $A$-torsion free quotient of $\pitilde (\{r_v\}_{v\in S})\otimes_A A/\fraka_i$. Let $W_\frakp$ denote the image of 
the diagonal map 
$$\kappa(\frakp) \otimes_A \pitilde (\{r_v\}_{v\in S}) \to \prod_i \kappa(\frakp) \otimes_{A/\fraka_i} V_i.$$
Suppose that the $A[G]$-module $\pitilde(\{r_v\}_{v\in S})$ exists. 
Then there is a $\kappa(\frakp)$-linear $G$-equivariant surjection 
$$\varsigma_\frakp:\bigotimes_{v\in S} \pitilde(r_{v, \frakp}) \to W_\frakp.$$
Moreover, if $\fraka$ is a minimal prime ideal of $A$ contained in $\frakp$ such that 
the rank of $N_v(\fraka)^j$ is equal to the rank of $(N_v(\fraka) \otimes_{A/\fraka} \kappa(\frakp))^j$ for 
all $j\geq 1$ and for any $v\in S$, then the surjection $\varsigma_\frakp$ is an isomorphism. 
Furthermore, when $s$ is equal to one, there is a $\kappa(\frakp)$-linear $G$-equivariant surjection 
$$\gamma_\frakp: \bigotimes_{v\in S} \widetilde \pi (r_{v, \frakp}) 
\to 
\kappa (\frakp) \otimes_A 
\widetilde \pi (\{r_v \}_{v\in S})$$
and it is an isomorphism if 
the rank of $N_v(\fraka_1)^j$ is equal to the rank of $(N_v(\fraka_1) \otimes_{A/\fraka_1} \kappa(\frakp))^j$ for 
all $j\geq 1$ and for any $v\in S$.
\end{mytheorem}
\begin{proof}
It is the content of \cite[Theorem 6.2.5, 6.2.6]{EmertonHelmLLC}. 
\end{proof}

\subsection{Autmorphic types}
\label{SubSec: Automorphic Galois type}
\begin{mydefinition}\label{Defn: Automorphic types}
Let $(\rho, N)$ be a Frobenius-semisimple Weil-Deligne representation of $W_K$ over a $\bbQpbar$. 
Let $I, m_1, \cdots, m_I$ be positive integers and 
$r_1, \cdots, r_I$ be irreducible Frobenius-semisimple representations of $W_K$ over $\overline L$ such that 
$(\rho, N) \otimes_L \overline L$ is isomorphic to $\oplus_{i=1}^I \Sp_{m_i}(r_i)$. 
We define the 
\textnormal{automorphic representation type $\AT^\rep (\rec (\rho, N)  )$ of} $\rec(\rho, N)$ to be 
$$\AT^\rep(\rec(\rho, N)) 
= ((\rec(r_1), m_1), \cdots, (\rec(r_I), m_I))$$
and the 
\textnormal{automorphic type $\AT(\rec(\rho, N))$ of} $\rec(\rho, N)$ to be 
$$\AT(\rec(\rho, N))= ((\dim r_1, m_1), \cdots, (\dim r_I, m_I)).$$
\end{mydefinition}

Note that though automorphic representation type and automorphic 
type of $\rec(\rho, N)$ is defined using the `Galois data' $r_i, m_i$ attached to $(\rho, N)$, 
these can also be defined in terms of automorphic representations 
attached to $\rec(\rho, N)$. 
Thus these notions are automorphic in nature. 
In fact, if we use Bernstein-Zelevinsky classification 
\cite{BernsteinZelevinskyInducedReprReductive, ZelevinskyInducedReprReductive} 
to express $\rec(\rho, N)$ as the quotient of an induced representation 
attached to some intervals $[\pi_1, n_1], \cdots, [\pi_J, n_J]$ 
where $\pi_i$ is a supercuspidal representation of $\gln_{d_i}(K)$ 
(see \cite[\S 4.3]{RodierBourbaki} for details), 
then by the local Langlands correspondence 
(see \cite[\S 4.4]{RodierBourbaki} for instance), 
it follows that $I=J$ and up to some reordering, 
$\pi_i \simeq \rec (r_i), d_i = \dim r_i , n_i= m_i$ for all $1\leq i\leq I$.

\section{Purity for big Galois representations}
\label{Sec: big purity}
Let $K, \calR, \calT, \calV, \iota_p$ be as in \S \ref{SubSec: Set up}. 
Denote the fraction field of $\calR$ by $\calL$ and 
the algebraic closure of $\bbQ$ in $\calLbar$ by $\bbQ^\cl$. 
Notice that $\bbQ^\cl$ is contained inside $\calR^\intal[1/p]$. 
Then by proposition \ref{Prop: rationality in general}, 
there exist positive integers $m, t_1\leq \cdots \leq t_m$, 
unramified characters $\chi_1, \cdots, \chi_m:W_K\to 
(\calR^\intal)^\times$, 
irreducible Frobenius-semisimple representations 
$\rho_1:W_K\to \gln_{d_1}(\bbQ^\cl), \cdots, \rho_m: W_K\to 
\gln_{d_m}(\bbQ^\cl)$ with finite image such that 
\begin{equation}
\label{Eqn: Decomposition of WD of big V into Sp t r}
\WD(\calV)^\Frss
\simeq 
\bigoplus_{i=1}^m \Sp_{t_i} (\chi_i\otimes \rho_i)_{/\calLbar}. 
\end{equation}
Let $\lambda: \calR \to \bbQpbar$ be a $\bbZp$-algebra homomorphism and 
$\pi_\lambda$ be the automorphic representation $\rec(\iota_p(\WD(V_\lambda)^\Frss))$.

\begin{mytheorem}[Purity for big Galois representations]\label{Thm: purity big}
Suppose $V_\lambda$ is pure of weight $w$. 

\begin{enumerate}

\item The Weil-Deligne representations 
$\WD(V_\lambda)^\Frss$ and $\oplus_{i=1}^m \Sp_{t_i}(\lambda^\intal\circ (\chi_i \otimes \rho_i))_{/\bbQpbar}$ are isomorphic. 

\item The rank of no power of the monodromy of $\calT_p$ decreases after specializing at $\lambda$. 

\item The polynomial $\Eul(\calV, X)^\mo$ has coefficients in 
$\calR^\inte$ and its specialization under $\lambda$ is $\Eul(V_\lambda, X)^\mo$.

\item The automorphic representation type $\AT^\rep(\pi_\lambda)$ of $\pi_\lambda$ is equal to 
$$\left( ( \rec(\iota_p( \lambda^\intal(\chi_1\otimes \rho_1))), t_1), \cdots, ( \rec( \iota_p(\lambda ^\intal(\chi_m\otimes \rho_m))), t_m)  \right).$$
\item The automorphic type $\AT(\pi_\lambda)$ of $\pi_\lambda$ is equal to the unordered tuple \newline
$\{(\dim \rho_1, t_1), \cdots, (\dim \rho_m, t_m) \}$. 
\end{enumerate}
Moreover, for any field extension $\calK$ of $\bbQp$ and any $\bbZp$-algebra homomorphism $\mu: \calR \to \calK$ with $\lambda (\ker \mu)=0$, 
the Weil-Deligne representation $\WD(V_\mu\otimes_\calK\calKbar)^\Frss$ is isomorphic to $\oplus_{i=1}^m \Sp_{t_i}(\mu^\intal \circ (\chi_i \otimes \rho_i) )_{/\calKbar}$. 
\end{mytheorem}

\begin{proof}
Denote the representation $\chi_i \otimes \rho_i$ by $r_i$ and the multiset 
$\cup_{i=1}^m 
\{ \lambda^\intal \circ r_i, \lambda^\intal \circ (|\Art^\mo_K|_K r_i), \cdots, \lambda^\intal \circ ( |\Art^\mo_K|_K^{t_i-1}r_i) \}$ by $S$. 
Let $N\in \End_{\calR_p}(\calT_p)$ be the monodromy of $\calT_p$. 
Note that conditions (A), (B), (C) below hold with $D= t_m$ 
(by equation \eqref{Eqn: Decomposition of WD of big V into Sp t r}). 
\begin{enumerate}[(A)]
 \item $\WD(V_\lambda)^\Frss$ is pure of weight $w$,
\item $\lambda^\intal \circ \tr \WD(\calV)^\Frss = \tr \WD(V_\lambda)^\Frss$, 
\item $\WD(V_\lambda)^\Frss$ is annihilated by the $D$-th power of its monodromy where 
$D$ denotes the size of $\WD(\calV)^\Frss$.
\end{enumerate}
The indecomposable summands of $\WD(V_\lambda)^\Frss$ are of size (see definition \ref{Defn: size}) at most $t_m$ by condition (C) and 
are of weight $w$ by condition (A) and remark \ref{Remark: Central summands determines purity}. 
Since the elements of $S$ are irreducible Frobenius-semisimple $W_K$-representations (by lemma \ref{Lemma: irreducible Frobenius semisimple under specializations}) and the sum of their traces is equal to $\tr \WD(V_\lambda)^\Frss$ (by condition (B)), 
the difference of the weights of 
any two elements of the multiset $S$ is at most 
$2(t_m-1)$. Note that 
the difference of the weights of $\lambda^\intal(r_m ), \lambda^\intal(|\Art_K^\mo|_K^{t_m-1}r_m)$ is $2(t_m-1)$. 
So these are a highest weight and a 
lowest weight element of $S$ respectively. 
By condition (A), $w$ is equal to the average of the weights of a highest 
weight and a lowest weight element of $S$, \ie, 
the average of the weights of 
$\lambda^\intal(r_m )$ and $\lambda^\intal(|\Art_K^\mo|_K^{t_m-1}r_m)$. 
So $\lambda^\intal(r_m)$ has weight $w+t_m-1$. 
Since $\lambda^\intal(r_m)$ is a highest weight element of $S$ 
and $\WD(V_\lambda)^\Frss$ is pure of weight $w$ (by condition (A)), 
the Weil-Deligne representation $\Sp_{t_m}(\lambda^\intal  (r_m))$ is a direct summand of 
$\WD(V_\lambda)^\Frss$. 
Now suppose that for an integer $1\leq m' < m$, 
the representation 
$\Sp_{t_{m'+1}}(\lambda^\intal \circ r_{m'+1}) \oplus \cdots \oplus \Sp_{t_m}(\lambda^\intal \circ r_m)$ is a 
direct summand of $\WD(V_\lambda)^\Frss$ as Weil-Deligne representations, \ie,  
there is an isomorphism
\begin{equation}\label{Eqn: Decomposition WD V lambda into W etc}
\WD(V_\lambda)^\Frss\simeq W\oplus\bigoplus_{i={m'+1}}^m \Sp_{t_i}(\lambda^\intal \circ r_i).
\end{equation}
Let $\calW$ denote the Weil-Deligne representation $\oplus _{i=1}^{m'} \Sp_{t_i}(r_i)$. Then 
the sum 
$\sum_{i=m'+1}^m (t_i-t_{m'})\dim \rho_i$ 
is equal to the integer 
$\dim _\calLbar N^{t_{m'}}(\WD(\calV)^\Frss)$ 
(by equation \eqref{Eqn: Decomposition of WD of big V into Sp t r}), 
which is larger than 
$\dim_{\bbQpbar} \lambda (N)^{t_{m'}} (\WD(V_\lambda)^\Frss)$ 
and this is bigger than 
$\dim_{\bbQpbar} \lambda (N)^{t_{m'}} W + \sum_{i=m'+1}^m (t_i-t_{m'}) \dim \rho_i$
(by equation \eqref{Eqn: Decomposition WD V lambda into W etc}). 
So $\lambda (N)^{t_{m'}} (W)=0$. 
Thus conditions (A'), (B'), (C') below hold with $D'=t_{m'}$. 
\begin{enumerate}[(A')]
\item $W$ is pure of weight $w$, 
\item $\lambda^\intal \circ \tr \calW = \tr W$, 
\item $W$ is annihilated by the $D'$-th power of its monodromy where $D'$ denotes the size of $\calW$. 
\end{enumerate}
Using an argument analogous to the proof of the fact that $\Sp_{t_m}(\lambda^\intal \circ r_m)$ is a 
direct summand of $\WD(V_\lambda)^\Frss$, we deduce that 
the Weil-Deligne representation $\Sp_{t_{m'}}(\lambda^\intal \circ r_{m'})$ is a direct summand of $W$. 
Then equation \eqref{Eqn: Decomposition WD V lambda into W etc} shows that 
$\Sp_{t_{m'}}(\lambda^\intal \circ r_{m'}) \oplus \Sp_{t_{m'+1}}(\lambda^\intal \circ r_{m'+1}) \oplus \cdots \oplus \Sp_{t_m}(\lambda^\intal \circ r_m)$ is a 
direct summand of $\WD(V_\lambda)^\Frss$. This proves part (1) by induction. Then part (2) to (5) follows. 

To simplify notations, we assume that $\calK$ is algebraically closed. 
Let $\calO_\mu$ (resp. $\calO_\lambda$) denote the image of $\mu$ (resp. $\lambda$) 
and $\eta: \calO_\mu\to \calO_\lambda$ denote the $\bbZp$-algebra homomorphism such that $\lambda = \eta \circ \mu$. 
Let $\lambda^\dag$ denote the map $\eta^\intal \circ\mu^\intal$. 
By proposition \ref{Prop: rationality in general}, 
there exist positive integers $M, t_1'\leq \cdots \leq t_M'$ and irreducible Frobenius-semisimple representations $s_1, \cdots, s_M$ over $\calO_\mu^\intal[1/p]$ such that 
$\WD(V_\mu)^\Frss$ is isomorphic to $\oplus_{i=1}^M \Sp_{t_i'} (s_i)$. 
By part (1), $\WD(V_\lambda)^\Frss$ is isomorphic to 
$\oplus_{i=1}^M \Sp_{t_i'}(\eta^\intal \circ s_i)$. 
Hence $M=m$ and $t_i'=t_i$ for all $1\leq i\leq M$. 
So $\eta^\intal \circ s_i, \lambda^\dag \circ r_i$ are of weight $w+t_i-1$ for all $1\leq i\leq m$. 
Note that for some integer $1\leq j\leq m$ and $0\leq a\leq t_j-1$, the representations $\mu^\intal\circ r_m$ and $s_j|\Art_K^\mo|_K^a$ are isomorphic. 
So the representations $\lambda^\dag \circ r_m, \eta^\intal \circ (s_j|\Art_K^\mo|_K^a)$
are of equal weight. This shows $t_m=t_j-2a$ and hence $a=0, t_j=t_m$. Thus $\Sp_{t_m}(\mu^\intal \circ r_m)$ is a direct summand of 
$\WD(V_\mu)^\Frss$. 
Now suppose that for an integer $1\leq m'<m$, the representation $\oplus_{i=m'+1}^m \Sp_{t_i}(\mu^\intal \circ r_i)$ is a direct summand of $\WD(V_\mu)^\Frss$. 
So by proposition \ref{Prop: rationality in general}, there exist irreducible Frobenius-semisimple representations $s_1', \cdots, s_{m'}'$ over $\calO_\mu^\intal[1/p]$  such that 
$\WD(V_\mu)^\Frss$ is isomorphic to $\bigoplus_{i=1}^{m'} \Sp_{t_i}(s_i') \oplus \bigoplus_{i=m'+1}^m \Sp_{t_i}(\mu^\intal \circ r_i)$. 
By part (1), $\WD(V_\lambda)^\Frss$ is isomorphic to 
$\bigoplus_{i=1}^{m'} \Sp_{t_i}(\eta^\intal\circ s_i') \oplus \bigoplus_{i=m'+1}^m \Sp_{t_i}(\eta^\intal \circ \mu^\intal \circ r_i)$.  
So $\eta^\intal \circ s_i', \lambda^\dag \circ r_i$ are of weight $w+t_i-1$ for all $1\leq i\leq m'$. 
Note that for some integer $1\leq k\leq m'$ and $0\leq b\leq t_k-1$, the representations $\mu^\intal\circ r_{m'}$ and $s_k'|\Art_K^\mo|_K^b$ are isomorphic. 
So the representations $\lambda^\dag \circ r_{m'}, \eta^\intal \circ (s_k'|\Art_K^\mo|_K^b)$
are of equal weight. This shows $t_{m'}=t_k-2b$ and hence $b=0, t_k=t_{m'}$. Thus $\Sp_{t_{m'}}(\mu^\intal \circ r_{m'})$ is a direct summand of 
$\oplus_{i=1}^{m'} \Sp_{t_i}(s_i')$ and hence $\oplus_{i=m'}^m \Sp_{t_i}(\mu^\intal \circ r_i)$ is a direct summand of $\WD(V_\mu)^\Frss$. This completes the proof by induction. 

\end{proof}

\section{Purity for pseudorepresentations}
\label{Sec: Purity for pseudorepresentations}

Let $\scrO$ be an integral domain containing $\bbZp$ as a subalgebra. 
We denote its fraction field by $\scrL$. 

\subsection{Preliminaries}
\label{SubSubSec: Preliminaires}
Let $\calO_1, \calO_2$ be integral domains containing $\bbZp$ as a subalgebra. 
We denote their fraction fields by $\calL_1, \calL_2$ respectively. 
Let $\res_1:\scrO \hra \calO_1$, $\res_2:\scrO\hra \calO_2$ be injective $\bbZp$-algebra homomorphisms. 
Let $T_0: W_K \to \scrO^\intal[1/p]$ be a pseudorepresentation of dimension $d\geq 1$ and 
$(r_1, N_1):W_K \to \gln_d(\calO_1^\intal[1/p])$, 
$(r_2, N_2): W_K \to \gln_d(\calO_2^\intal[1/p])$ be 
Weil-Deligne representations such that 
\begin{equation}
\label{Eqn: tr r U r V}
\res_1^\intal \circ T_0 = \tr(r_1), \quad 
\res_2^\intal \circ T_0 = \tr(r_2).
\end{equation}
Suppose that there exist $\bbZp$-algebra homomorphisms $f_1: \calO_1^\intal \to \bbQpbar$, 
$f_2: \calO_2^\intal \to \bbQpbar$ such that 
$f_1\circ (r_1, N_1), f_2\circ (r_2, N_2)$ are pure. 
We first state two propositions. Then we prove a lemma which will be used to establish these propositions. 
For the notion of size, we refer to definition \ref{Defn: size}. 
\begin{myproposition}
\label{Prop: size smaller}
The size of $(f_1 \circ (r_1, N_1))^\Frss$ is smaller than the 
size of $(f_2\circ (r_2, N_2))^\Frss$. 
Consequently, these two representations have the same size. 
\end{myproposition}

Let $\kappa, t_1\leq  \cdots \leq t_\kappa$ be positive integers 
and $\theta_{11}, \cdots, \theta_{1t_1}$, 
$\theta_{21}, \cdots, \theta_{2t_2}$, $\cdots$, 
$\theta_{\kappa 1}, \cdots, \theta_{\kappa t_\kappa}$ be 
irreducible Frobenius-semisimple representations of $W_K$ over $\scrO^\intal[1/p]$ 
such that 
\begin{enumerate}
\item $T_0$ is equal to $\sum_{i=1}^\kappa \sum_{j=1}^{t_\kappa} \tr \theta_{ij}$, 
\item for any $1\leq i\leq \kappa, 1\leq j\leq t_i$, the representations 
$\res^\intal_1 \circ  \theta_{ij}$, 
$\res^\intal_1 \circ ( |\Art^\mo_K|_K^{j-1} \theta_{i1})$ of $W_K$ are isomorphic 
over $\calLbar_1$ and 
\item there is an isomorphism 
\begin{equation}
\label{Eqn: decomposition of rho U tilde}
((r_1, N_1) \otimes_{\calO_1} \calLbar_1 )^\Frss
\simeq 
\bigoplus_{i=1}^\kappa \Sp_{t_i} (\res^\intal_1\circ \theta_{i1}) .
\end{equation}
\end{enumerate}

\begin{myproposition}
\label{Prop: obtain summand}
The representation 
$\Sp_{t_\kappa} (\res^\intal_2\circ \theta_{\kappa 1})$ 
is a direct summand of $((r_2, N_2) \otimes_{\calO_2} \calLbar_2)^\Frss$ 
as Weil-Deligne representations. 
\end{myproposition}

\begin{mylemma}
\label{Lemma: Eqn lemma}
Let $k, s_1 \leq  \cdots \leq s_k$ be positive integers 
and $\vartheta_1, \cdots, \vartheta_k$ be irreducible Frobenius-semisimple 
representations of $W_K$ over $\calO_2^\intal[1/p]$ such that 
\begin{equation}
\label{Eqn: decomposition of rho V tilde}
((r_2, N_2) \otimes_{\calO_2} \calLbar_2) ^\Frss 
\simeq 
\bigoplus_{i=1}^k \Sp_{s_i} (\vartheta_i).
\end{equation}
Then for some integers $1\leq a, b \leq k$, we have 
\begin{equation}
\label{Eqn: top bottom general}
(\res^\intal_2 \circ \theta_{\kappa t_\kappa})_{/\calLbar_2} \simeq (\vartheta_a |\Art^\mo_K|_K ^{s_a-1})_{/\calLbar_2}, 
\quad
(\res^\intal_2 \circ  \theta_{\kappa 1} )_{/\calLbar_2}\simeq (\vartheta_b)_{/\calLbar_2}, 
\end{equation}
\begin{equation}
\label{Eqn: inequalities general}
2t_\kappa= s_a+s_b\leq 2s_k.
\end{equation}

\end{mylemma}

\begin{proof}
By lemma \ref{Lemma: irreducible Frobenius semisimple under specializations}, 
$\res_1^\intal \circ \theta_{i1}$ is an irreducible Frobenius-semisimple representation 
of $W_K$ over $\calO_1^\intal[1/p]$. 
Since $f_1 \circ (r_1, N_1)$ is pure, theorem \ref{Thm: purity big} and 
equation \eqref{Eqn: decomposition of rho U tilde} give 
\begin{equation}
\label{Eqn: rho y general}
(f_1 \circ (r_1, N_1))^\Frss \simeq 
\bigoplus_{i=1}^\kappa \Sp_{t_i} (f_1 \circ \res^\intal_1\circ \theta_{i1}).
\end{equation}
Since $t_1\leq \cdots\leq t_\kappa$ and $f_1\circ (r_1, N_1)$ is pure, 
by equation \eqref{Eqn: rho y general}, 
no eigenvalue of $\phi$ on $f_1 \circ (r_1, N_1)$ has weight strictly more (resp. less) than 
the weight of the $\phi$-eigenvalues on $\theta_{\kappa 1}$ (resp. $\theta_{\kappa t_\kappa}$). 
So there are no integers $i, j$ with $1\leq i\leq \kappa, 1\leq j\leq t_i$ 
such that 
$\theta_{ij}$ is isomorphic to $\theta_{\kappa 1}|\Art^\mo_K|_K ^{-\nu}$ or $\theta_{\kappa t_\kappa}|\Art^\mo_K|_K^\nu$ for some integer $\nu\geq 1$. 
Note that by lemma \ref{Lemma: irreducible Frobenius semisimple under specializations}, 
there exist integers $1\leq a, b\leq k$ such that 
the $W_K$-representation 
$\res^\intal_2\circ \theta_{\kappa \alpha_\kappa}$ (resp. $\res^\intal_2\circ \theta_{\kappa 1}$) 
is isomorphic to $\vartheta_a|\Art^\mo_K|_K ^{j_1}$ (resp. $\vartheta_b|\Art^\mo_K|_K ^{j_2}$) 
over $\calLbar_2$ where $0\leq j_1\leq s_a-1$ (resp. $0\leq j_2\leq s_b-1$). 
Now for some $1\leq i\leq \kappa, 1\leq j\leq t_i$, 
the $W_K$-representations $\res^\intal_2\circ \theta_{ij}, \vartheta_a|\Art^\mo_K|_K ^{s_a-1}
= (\res^\intal_2\circ \theta_{\kappa \alpha_\kappa})|\Art^\mo_K|_K ^{s_a-1-j_1}$ 
are isomorphic over $\calLbar_2$. 
As $\res_2$ is injective and the traces of the representations  
$\theta_{ij}$ and $|\Art^\mo_K|_K ^{s_a-1-j_1} \theta_{\kappa \alpha_\kappa}$ coincide after 
composing them with $\res^\intal_2$, 
these representations are isomorphic over $\scrLbar$ (by \cite[Chapter 1, \S 2]{SerreAbelianEllAdic} for instance).
As noted before, $s_a-1-j_1$ cannot be positive. So 
$j_1$ is equal to $s_a-1$. Similarly $j_2$ is zero. Thus  
equation \eqref{Eqn: top bottom general} holds. 

Let $w$ denote the weight of the pure representation $f_2\circ (r_2, N_2)$. 
By theorem \ref{Thm: purity big} and 
equation \eqref{Eqn: decomposition of rho V tilde}, 
\begin{equation}
\label{Eqn: rho z general}
(f_2 \circ (r_2, N_2))^\Frss \simeq 
\bigoplus_{i=1}^k \Sp_{s_i}(f_2 \circ \vartheta_i).
\end{equation}
So the weight of any $\phi$-eigenvalue on 
$f_2 \circ \vartheta_b$ (resp. $f_2\circ \vartheta_a |\Art^\mo_K|_K ^{s_a-1}$) 
is equal to $w+(s_b-1)$ (resp. $w-(s_a-1)$). 
So their difference, denoted $\delta$, is equal to $s_a+s_b-2$. 
On the other hand, since $\theta_{\kappa t_\kappa}$ and $|\Art^\mo_K|_K^{t_\kappa-1}\theta_{\kappa 1}$ 
are isomorphic over $\scrLbar$ (as their traces become equal after composing them with $\res^\intal_1$ and $\res_1$ is injective), by 
equation \eqref{Eqn: top bottom general}, $\delta$ is equal to $2(t_\kappa-1)$. 
Since $s_a, s_b$ are smaller than $s_k$, we get equation \eqref{Eqn: inequalities general}. 
\end{proof}

\begin{proof}[Proof of proposition \ref{Prop: size smaller}]
Equation \eqref{Eqn: inequalities general}, \eqref{Eqn: rho y general}, 
\eqref{Eqn: rho z general} give the first part of proposition 
\ref{Prop: size smaller}. Then the second part follows. 
\end{proof}

\begin{proof}[Proof of proposition \ref{Prop: obtain summand}]
By proposition \ref{Prop: size smaller}, $t_\kappa$ is equal to $s_k$. 
Then equation \eqref{Eqn: inequalities general} gives $s_a=s_b=s_k$. 
So 
$\Sp_{s_b}(\vartheta_b) = \Sp_{s_k}(\res^\intal_2\circ  \theta_{\kappa 1})$ 
is a direct summand of $((r_2, N_2) \otimes_{\calO_2} \calLbar_2)^\Frss$. 
\end{proof}

\subsection{Pseudorepresentations of Weil groups}
\label{SubSubSec: proof theorem 1 5}

Let $A$ be a commutative ring and $R$ be an $A$-algebra. 
Given a pseudorepresentation $T:R \to A$ of dimension $d\geq 1$, 
the degree $d$ monic polynomial 
$P_{x, T}(X)=X^d+(-1)^{d-1}T(x) X^{d-1}+\cdots$ 
(as defined in \cite[\S 1.2.3]{BellaicheChenevierAsterisQUE}) is called 
the characteristic polynomial of $x$ (for $T$). It has coefficients in $A[1/d!]$. 

\begin{mytheorem}
[Purity for pseudorepresentations]
\label{Thm: purity for pseudorepresentations}
Let $\calO$ be an integral domain over $\bbZp$ 
and $\res:\scrO \hra\calO$ be an injective $\bbZp$-algebra homomorphism. 
Let $T: W_K \to \scrO$ be a pseudorepresentation of dimension $n\geq 1$ and let 
$(r, N):W_K \to \gln_n(\calO[1/p])$ be a 
Weil-Deligne representation such that $\res\circ T = \tr r$. 
Suppose $f \circ (r, N)$ is pure for some $\bbZp$-algebra homomorphism $f: \calO \to \bbQpbar$. 
Then there exist positive integers 
$m, t_1\leq t_2\leq \cdots \leq t_m$ and 
irreducible Frobenius-semisimple representations 
$r_1, \cdots, r_m$ of $W_K$ with coefficients in $\scrO^\intal[1/p]$ such that 
the statements (1), (2), (3) hold. 
\begin{enumerate}
\item $T$ is equal to $\sum_{i=1}^m \sum _{j=1}^{t_i} \tr r_i |\Art^\mo_K|_K^{j-1}$. 
\item If there exist an integral domain $\calO'$ over $\bbZp$ 
and a Weil-Deligne representation $(r', N'): W_K\to \gln_n(\calO'[1/p])$ such that 
\begin{itemize}
\item $\res'\circ T = \tr r'$ for some injective $\bbZp$-algebra homomorphism $\res':\scrO \hra \calO'$, 
\item 
$f'\circ (r', N')$ is pure for some $\bbZp$-algebra homomorphism $f': \calO' \to \bbQpbar$, 
\end{itemize}
then for any lift $\res'^\dag$ of $\res'$ and any lift $f'^\dag$ of $f'$, there are isomorphisms 
\begin{align}
((r', N')\otimes_{\calO'}\overline Q(\calO') )^\Frss
&\simeq 
\bigoplus_{i=1}^m \Sp_{t_i} (\res'^\dag \circ r_i), \label{Eqn: Isom 1}\\
(f'\circ (r', N'))^\Frss 
&\simeq \bigoplus_{i=1}^m \Sp_{t_i} (f'^\dag\circ \res'^\dag \circ r_i).\label{Eqn: Isom 2}
\end{align}
\item If the characteristic polynomial $P_{\phi, T}(X)$ of $\phi$ has coefficients in $\scrO^\intal\cap \scrO[1/n!]$, then 
$r_i$ has values in $\scrO^\intal$ whenever $r_i$ is a character for some $1\leq i\leq m$. 
\end{enumerate}
Moreover, if there are positive integers $M, s_1, \cdots, s_M$ and 
irreducible Frobenius-semisimple representations $R_1, \cdots, R_M$ of $W_K$ 
over $\scrO^\intal[1/p]$ such that the statements (1), (2) above hold 
(when $m, t_i, r_i$ are replaced by $M, s_i, R_i$ respectively), then 
$m=M, t_1=s_1, \cdots, t_m=s_M$ and 
there exists a permutation $\sigma$ on $\{1, \cdots, m\}$ such that 
\begin{enumerate}[(i)]
\item $r_{\sigma(i)}$ is isomorphic to $R_i$ over $\scrLbar$ for all $1\leq i\leq m$, 
\item $\{a, a+1, \cdots, b\}$ is stable under the action of $\sigma$ 
whenever $t_{a-1}<t_a=\cdots=t_b<t_{b+1}$ 
for some integers $1\leq a, b \leq m$ (here $t_0:=0, t_{m+1}:=t_m+1$). 
\end{enumerate}
\end{mytheorem}

\begin{proof}
Let $\calL$ denote the fraction field of $\calO$. 
By proposition \ref{Prop: rationality in general}, 
there exist positive integers $m$, $t_1\leq t_2 \leq \cdots \leq t_m$ and 
irreducible Frobenius-semisimple representations 
$\tau_1, \cdots, \tau_m$ of $W_K$ with coefficients in 
$\calO^\intal[1/p]$ such that 
$((r, N)\otimes_{\calO}\calLbar )^\Frss$ is isomorphic to $\oplus_{i=1}^m \Sp_{t_i}(\tau_i)$. 
Since $\tr r= \res\circ T$, the characteristic polynomial of 
$\tau_i$ has coefficients in $(\res \scrO)^\intal[1/p]$ 
(we consider $Q(\res \scrO)$ as a subfield of $\calL$ and 
thus $(\res \scrO)^\intal$ is a subring of $\calO^\intal$). 
So by proposition \ref{Prop: rationality in general}, 
we may (and do) assume that 
$\tau_i$ has coefficients in $(\res \scrO)^\intal[1/p]$. 
So there exist irreducible Frobenius-semisimple 
representations $r_1, \cdots, r_m$ of $W_K$ 
with coefficients in $\scrO^\intal[1/p]$ such that 
$\res \circ r_1=\tau_1, \cdots, \res \circ r_m=\tau_m$. 
Since 
$T- \sum_{i=1}^m \sum _{j=1}^{t_i} \tr r_{i}|\Art^\mo_K|_K^{j-1}$ 
goes to zero under $\res^\intal$ 
and $\res$ is injective, we get part (1).

Let $\calL'$ denote the fraction field of $\calO'$. 
By proposition \ref{Prop: obtain summand}, 
$\Sp_{t_m}(\res'^\dag \circ r_{m})$ is a direct summand of 
$((r', N') \otimes_{\calO'} \calLbar')^\Frss$. 
Suppose for some $1\leq k< m$, 
$\oplus_{i=k+1}^m \Sp_{t_i}(\res'^\dag \circ r_{i})$ 
is a direct summand of $((r', N')\otimes_{\calO'}\calLbar')^\Frss $. 
We will now show that 
$\oplus_{i=k}^m \Sp_{t_i}(\res'^\dag \circ r_{i})$ 
is a direct summand of $((r', N')\otimes_{\calO'}\calLbar')^\Frss $. 
By proposition \ref{Prop: rationality in general}, there exist
positive integers $Q, s_1\leq \cdots \leq s_Q$ and 
irreducible Frobenius-semisimple 
representations $\eta_1, \cdots, \eta_Q$ 
of $W_K$ with coefficients in $\calO'^\intal[1/p]$ such that 
$((r', N')\otimes_{\calO'}\calLbar')^\Frss$ 
is isomorphic to 
$\bigoplus_{i=1}^ Q 
\Sp_{s_i}(\eta_i)
\oplus
\bigoplus_{i=k+1}^m \Sp_{t_i} (\res'^\dag\circ r_{i})$. 
Note that the specialization of the pseudorepresentation 
$\sum _{i=1}^k \sum_{j=1}^{t_i} \tr r_{i}|\Art^\mo_K|_K^{j-1}:W_K\to \scrO^\intal[1/p]$ 
under $\res^\intal$ (resp. $\res'^\dag$) 
is equal to the trace of the Weil-Deligne representation 
$\oplus_{i=1}^k \Sp_{t_i}(\tau_i)$ 
(resp. $\oplus_{i=1}^Q \Sp_{s_i}(\eta_i)$) 
of $W_K$ with coefficients in 
$\calO^\intal[1/p]$ (resp. $\calO'^\intal [1/p]$). 
So by proposition \ref{Prop: obtain summand}, 
the representation  
$\Sp_{t_k}(\res'^\dag \circ r_k)$ is a direct summand of 
$\oplus_{i=1}^ Q 
\Sp_{s_i}(\eta_i)$. 
This shows that 
$\oplus_{i=k}^m \Sp_{t_i}(\res'^\dag \circ r_{i})$ 
is a direct summand of $((r', N')\otimes_{\calO'}\calLbar')^\Frss $. 
So we obtain equation \eqref{Eqn: Isom 1} by induction. 
Using theorem \ref{Thm: purity big}, we get 
equation \eqref{Eqn: Isom 2}. 
Part (3) is clear. 

To establish the final part, note that 
$((r, N)\otimes_\calO \calLbar)^\Frss$ 
is isomorphic to 
$\oplus_{i=1}^m \Sp_{t_i}(\res^\intal\circ r_i)$ and 
$\oplus_{i=1}^M \Sp_{s_i}(\res^\intal\circ R_i)$. 
This shows that $m=M, t_1=s_1, \cdots, t_m=s_M$. 
By theorem \ref{Thm: Structure of Frob ss Weil-Deligne representations}, 
there exists a permutation $\sigma$ on $\{1, \cdots, m\}$ such that 
condition (ii) above holds and 
$\res^\intal\circ r_{\sigma(i)}$ is isomorphic to $\res^\intal \circ R_i$ for all 
$1\leq i\leq m$. 
Since $\res$ is injective, $\res^\intal$ is also injective. So $r_{\sigma(i)}$ and $R_i$ 
have same traces and hence these are isomorphic over $\scrLbar$ 
(by \cite[Chapter 1, \S 2]{SerreAbelianEllAdic} for instance).
\end{proof}

\subsection{Pure specializations of pseudorepresentations of global Galois groups}
\label{SubSec: purity sum}
Given a local ring $(A, \frakm)$, we denote its Henselization by $(A^h, \frakm^h)$ 
(see \cite[\href{http://stacks.math.columbia.edu/tag/04GQ}{Tag 04GQ}]{StacksProject}) 
and consider their residue fields to be equal via the isomorphism $A/\frakm \to A^h/\frakm^h$ 
(see \cite[\href{http://stacks.math.columbia.edu/tag/04GN}{Tag 04GN}]{StacksProject}). 
Since the map $A\to A^h$ is flat (by 
\cite[\href{http://stacks.math.columbia.edu/tag/07QM}{Tag 07QM}]{StacksProject} for instance) 
and flat maps satisfy going down property 
(see \cite[\href{http://stacks.math.columbia.edu/tag/00HS}{Tag 00HS}]{StacksProject}), 
the minimal primes of $A^h$ go to the minimal primes of $A$ under the inverse of the map $A\to A^h$. 
Given a prime ideal $\frakp$ of a ring $R$, the mod $\frakp$ reduction map is denoted by $\pi_\frakp$. 

Let $F$ be a number field and 
$T:G_F \to \scrO$ be a pseudorepresentation such that 
$T=T_1+\cdots+T_n$ where $T_1:G_F \to \scrO, \cdots, T_n:G_F\to \scrO$ are traces of irreducible 
representations $\sigma_1, \cdots, \sigma_n$ of $G_F$ over $\scrLbar$. 
Let $w\nmid p$ be a finite place of $F$ and assume that 
the restrictions of $\sigma_1, \cdots, \sigma_n$ to $W_w$ are monodromic. 

\begin{mydefinition}
\label{Defn: locus}
The \textnormal{irreducibility and $w$-purity locus} 
(\textnormal{irreducibility and purity locus}, in short) of $T_1, \cdots, T_n$ 
is defined to be the 
collection of all tuples of the form $(\calO, \frakp, \kappa, \res, \rho_1$, $\cdots, \rho_n)$ 
where $\calO$ is a domain over $\bbZp$, 
$\frakp$ is a prime ideal of $\calO$ such that the Henselization $\calO_\frakp^h$ of $\calO_\frakp$ is Hausdorff, 
$\kappa$ denotes the residue field $\calO_\frakp/\frakp \calO_\frakp$ and is an algebraic extension of $\bbQp$, 
$\res:\scrO \hra\calO$ is an injective $\bbZp$-algebra homomorphism 
and for each $1\leq i\leq n$, 
$\rho_i$ is an irreducible $G_F$-representation over $\overline \kappa$ such that 
$\pi_\frakp \circ \res \circ T_i$ is equal to the trace of $\rho_i$ and $\rho_i|_{G_w}$ is pure 
(of some weight depending on $i$). 
\end{mydefinition}

For each element $(\calO, \frakp, \kappa, \res, \rho_1, \cdots, \rho_n)$ of this locus, 
we choose semisimple $G_F$-representations $\widetilde \rho_1$, $\cdots$, $\widetilde \rho_n$ 
over $\overline Q(\calO)$ 
such that $\tr \widetilde \rho_i= \res \circ T_i$ for all $1\leq i\leq n$ 
(using \cite[Theorem 1]{TaylorGaloisReprAssociatedToSiegelModForms}) 
and choose $G_F$-representations 
$\varrho_1, \cdots, \varrho_n$ over $\calO_\frakp^h$ such that 
$\tr \varrho_i= \res \circ T_i$ for all $1\leq i\leq n$ 
(using \cite[Th\'eor\`eme 1]{NyssenPseudoRepresentations} and the fact that $\calO_\frakp^h$ is 
Hausdorff). 
We also fix a minimal prime $\fraka$ of $\calO_\frakp^h$. 
The composite maps $\calO \to \calO/\frakp \to \kappa$, 
$\calO \to \calO_{\frakp} \to \calO_{\frakp}^h$ 
and $\calO \to \calO_{\frakp} \to \calO_{\frakp}^h \to \calO_{\frakp}^h/\fraka$ 
are denoted by $\pi_\frakp$, $h_\frakp$ and $\pi_\fraka \circ h_\frakp$ respectively. 
Note that the map $\pi_\fraka \circ h_\frakp$ is injective (as observed in the beginning of 
\S \ref{SubSec: purity sum}).

\begin{mytheorem}
\label{Thm: purity sum}
Suppose that the irreducibility and purity locus of $T_1, \cdots, T_n$ is nonempty. Then there exist positive integers 
$m, t_1\leq t_2\leq \cdots \leq t_m$ and 
irreducible Frobenius-semisimple representations 
$r_1, \cdots, r_m$ of $W_w$ with coefficients in $\scrO^\intal[1/p]$ such that 
the following hold.
\begin{enumerate}
\item $T|_{W_w}$ is equal to $\sum_{i=1}^m \sum _{j=1}^{t_i} \tr r_i |\Art^\mo_K|_K^{j-1}$. 
\item If $(\calO, \frakp, \kappa, \res, \rho_1, \cdots, \rho_n)$ is an element of the irreducibility and purity locus of $T_1$, $\cdots$, $T_n$, 
then for any lift $\res^\dag$ (resp. $\pi^\dag_{\frakp}, (\pi_{\fraka}\circ h_\frakp)^\dag$) of 
$\res$ (resp. $\pi_{\frakp}, \pi_{\fraka} \circ h_\frakp$), 
there are isomorphisms 
\begin{align}
\WD \left(\bigoplus_{i=1}^n \rho_i |_{W_w} \right)^\Frss
& \simeq 
\bigoplus_{i=1}^m \Sp_{t_i} (\pi_{\frakp}^\dag\circ \res^\dag \circ r_i), \label{Eqn: purity sum rho prime}\\
\WD \left(\bigoplus_{i=1}^n (\pi_{\fraka} \circ \varrho_i) \otimes \Qbar(\calO_\frakp^h/\fraka) |_{W_w} \right)^\Frss 
& \simeq 
\bigoplus_{i=1}^m \Sp_{t_i} ( (\pi_{\fraka} \circ h_\frakp)^\dag \circ  \res^\dag \circ r_i), \label{Eqn: purity sum varrho} \\
\WD \left(\bigoplus_{i=1}^n \widetilde \rho_i |_{W_w} \right)^\Frss
&\simeq 
\bigoplus_{i=1}^m \Sp_{t_i} ( \res^\dag \circ r_i)\label{Eqn: purity sum tilde rho}, \\
\WD \left(\bigoplus_{i=1}^n \sigma_i |_{W_w} \right)^\Frss
&\simeq 
\bigoplus_{i=1}^m \Sp_{t_i} (r_i).\label{Eqn: purity sum sigma}
\end{align}
\end{enumerate}
\end{mytheorem}

\begin{proof}
Since $(\pi_\fraka \circ \varrho_i) \otimes \overline Q(\calO_\frakp^h/\fraka)$ is irreducible 
and the $G_F$-representations $\widetilde \rho_i\otimes \overline Q(\calO_\frakp^h/\fraka)$, 
$(\pi_\fraka \circ \varrho_i) \otimes \overline Q(\calO_\frakp^h/\fraka)$ have same traces, 
these are isomorphic. 
Similarly $\sigma_i\otimes \Qbar(\calO)$ and $\widetilde \rho_i$ are isomorphic. 
Since $\sigma_i|_{W_w}$ is monodromic, $\pi_\fraka \circ \varrho_i|_{W_w}$ is monodromic. 
Note that $\WD(\pi_\fraka \circ \varrho_i|_{W_w})$ has coefficients in 
$(\calO_\frakp^h/\fraka)[1/p]$, its trace is equal to 
$\pi_\fraka\circ  h_\frakp  \circ \res \circ T|_{W_w}$ 
and it has a pure specialization 
$\WD(\rho_i|_{W_w})$. 
Also note that the map $\pi_\fraka\circ  h_\frakp \circ \res$ is an injective $\bbZp$-algebra homomorphism 
from $\scrO$ to the domain $\calO_\frakp^h/\fraka$ over $\bbZp$. 
Then theorem \ref{Thm: purity for pseudorepresentations} gives 
part (1) and equation \eqref{Eqn: purity sum rho prime}, 
\eqref{Eqn: purity sum varrho}. 
Since $(\pi_\fraka \circ \varrho_i) \otimes \overline Q(\calO_\frakp^h/\fraka)$ 
is isomorphic to 
$\widetilde \rho_i\otimes \overline Q(\calO_\frakp^h/\fraka)$ 
and $\sigma_i\otimes \Qbar(\calO_\frakp^h/\fraka)$, 
we get equation \eqref{Eqn: purity sum tilde rho} and 
\eqref{Eqn: purity sum sigma} from equation \eqref{Eqn: purity sum varrho}. 
\end{proof}

\section{Local Langlands correspondence for $\gln_n$ in families}
\label{Sec: LLC families}
In this section, we use theorem \ref{Thm: purity big} to strenthen a part of  
the local Langlands correspondence for $\gln_n$ in families (see theorem \ref{Thm: EH 626}) formulated by 
Emerton and Helm. 
Let $S, G, r_v$ be as in theorem \ref{Thm: EH 621} and 
suppose that the $A[G]$-module $\pitilde(\{r_v\}_{v\in S})$ exists. 
\begin{mytheorem}
\label{Thm: LLC families}
Let $\frakp$ be a prime of $A[1/p]$. Suppose there exists a $\bbZp$-algebra homomorphism $i_\frakp : A\to \bbQpbar$ such that 
$\frakp$ is contained inside the kernel of $i_\frakp$ and 
$r_v \otimes_{A, i_\frakp} \bbQpbar$ is pure for all $v\in S$. 
Then the surjection $\varsigma_\frakp$ as in theorem \ref{Thm: EH 626} is an isomorphism. 
If $\frakp$ lies on only one irreducible component of $\Spec A[1/p]$, then 
the surjection $\gamma_\frakp$ as in theorem \ref{Thm: EH 626} is also an isomorphism. 
\end{mytheorem}

\begin{proof}
Let $\fraka$ denote a minimal prime of $A$ contained in $\frakp$. 
Then by theorem \ref{Thm: purity big}, 
the rank of the $i$-th power of monodromy of $r_{v, \fraka}$ 
is equal to the rank of the $i$-th power of the monodromy of $r_{v, \frakp}$ 
for any $i\geq 1$ and any $v\in S$.  
Hence the result follows from theorem \ref{Thm: EH 626}. 
\end{proof}

\section{Families of Galois representations}
\label{Sec: Constancy of automorphic types}
The goal of this section is to illustrate the role of purity for big Galois representations, 
purity for pseudorepresentations and theorem \ref{Thm: purity sum} in the study of variation of local 
Euler factors, local automorphic types, intersection points of irreducible components etc.~
for families of Galois representations.

\subsection{Hida families}
For Hida theory of ordinary cusp forms, we follow \cite{HidaICM86} and 
refer to the references \cite{HidaGalrepreord, HidaIwasawa} contained therein. 
We follow \cite{GeraghtyPhDThesis} for Hida theory for definite unitary groups.

\subsubsection{Cusp forms}
Let $f=\sum_{n=1}^\infty a_n(f)q^n$ be a normalized eigen cusp form of weight $k\geq 2$. 
Then by 
\cite{Eichler54,ShimuraCorrespondances} (for $k=2$), 
\cite{DeligneModFormAndlAdicRepr} (for $k>2$), there exists 
a unique (up to equivalence) continuous Galois representation $\rho_f:G_\bbQ \to \gln_2(\bbQpbar)$ such that 
$\tr \rho_f(\Fr_\ell)= a_\ell(f)$ for any prime $\ell$ not dividing $p$ and the level of $f$. 
Let $\pi(f)= \otimes_{\ell\leq \infty}'\pi(f)_\ell$ denote the 
irreducible unitary representation of $\gln_2(\bbA_\bbQ)$ corresponding to $f$ (see \cite[Theorem 5.19, 4.30]{GelbertRedBook}).

Let $N$ be a positive integer and $p$ be an odd prime with $p\nmid N$ and $Np\geq 4$. 
Let $h^\ord$ be the universal $p$-ordinary Hecke algebra of tame level $N$ 
(denoted $h^\ord(N; \bbZp)$ in \cite{HidaICM86}). 
It has an algebra structure over $\bbZp[[X]]$. 
Let $\fraka$ be a minimal prime of $h^\ord$. 
Let $\calR(\fraka)$ denote the ring $h^\ord/\fraka$ and 
$\calQ(\fraka)$ denote the fraction field of $\calR(\fraka)$. 
Let $\calQbar(\fraka)$ be an algebraic closure of $\calQ(\fraka)$.  
Let $S$ denote the set of places of $\bbQ$ dividing $Np\infty$. 
By \cite[Theorem 3.1]{HidaICM86}, 
there exists a unique (up to equivalence) continuous (in the sense of \cite[\S 3]{HidaICM86}) Galois representation 
$\rho_\fraka: G_{\bbQ, S}\to \gln_2(\calQ(\fraka))$ 
such that $\rho_\fraka$ has traces in $\calR(\fraka)$ and $\tr(\rho_\fraka(\Fr_\ell)) = T_\ell\modu \fraka$ for all prime $\ell\nmid Np$ 
where $T_\ell\in h^\ord$ denotes the Hecke operator associated to $\ell$. 
Henceforth, the representation $\rho_\fraka \otimes \calQbar(\fraka)$ is 
denoted by $\rho_\fraka$. 
A $\bbZp$-algebra homomorphism $\lambda:h^\ord \to \bbQpbar$ is 
said to be an {\it arithmetic specialization} if 
$\lambda((1+X)^{p^r} - (1+p)^{kp^r})=0$ for some integers $k\geq 0$ and $r\geq 0$. 
The arithmetic specializations of $h^\ord$ are in one-to-one correspondence 
(by the isomorphism of \cite[Theorem 2.2]{HidaICM86}) 
with the $p$-ordinary $p$-stabilized (in the sense of \cite[p.~538]{WilesOrdinaryLambdaAdic}) 
normalized eigen cusp forms of tame level a divisor of $N$. 
Given an arithmetic specialization $\lambda$ of $h^\ord$, let $f_\lambda$ denote the 
corresponding ordinary form.

\begin{mydefinition}
\label{Defn: automorphic type: cusp}
The \textnormal{automorphic type} of a minimal prime $\fraka$ of $h^\ord$ at a prime $\ell\neq p$ is defined to be the unordered tuple 
$\AT_\ell(\fraka)$ if the automorphic types of $\pi(f_\lambda)_\ell$ are equal to $\AT_\ell(\fraka)$ 
for all arithmetic specialization $\lambda$ of $h^\ord$ with $\lambda(\fraka)=0$. 
\end{mydefinition}

\begin{mytheorem}
\label{Thm: Application: Hida cusp}
Let $\fraka$ be a minimal prime of $h^\ord$ and $\ell\neq p$ be a prime. 
Then the following hold. 
\begin{enumerate}
\item If $\WD(\rho_\fraka|_{W_\ell})^\Frss$ is indecomposable and irreducible, 
then there exists an irreducible Frobenius-semisimple 
representation $r$ over $\calR(\fraka)^\intal[1/p]$ such that 
$\WD(\rho_\fraka|_{W_\ell})^\Frss$ is isomorphic to $r \otimes \calQbar (\fraka)$ and 
$\WD(\rho_{f_\lambda}|_{W_\ell})^\Frss$ is isomorphic to $\lambda^\intal \circ r$ 
for any arithmetic specialization $\lambda$ of $h^\ord$ with $\lambda(\fraka)=0$. 

\item If $\WD(\rho_\fraka|_{W_\ell})^\Frss$ is indecomposable and reducible, then 
there exists an $\calR(\fraka)^\intal$-valued character $\chi$ of $W_\ell$ such that 
$\WD(\rho_\fraka|_{W_\ell})^\Frss$ is isomorphic to $\Sp_2(\chi) \otimes 
\calQbar (\fraka)$ and 
$\WD(\rho_{f_\lambda}|_{W_\ell})^\Frss$ is isomorphic to $\lambda^\intal \circ \Sp_2(\chi)$ 
for any arithmetic specialization $\lambda$ of $h^\ord$ with $\lambda(\fraka)=0$. 
\item If $\WD(\rho_\fraka|_{W_\ell})^\Frss$ is decomposable, then 
there exist $\calR(\fraka)^\intal$-valued characters $\chi_1, \chi_2$ of $W_\ell$ such that 
$\WD(\rho_\fraka|_{W_\ell})^\Frss$ is isomorphic to $(\chi_1\oplus \chi_2) \otimes 
\calQbar (\fraka)$ and 
$\WD(\rho_{f_\lambda}|_{W_\ell})^\Frss$ is isomorphic to $\lambda^\intal \circ (\chi_1\oplus \chi_2)$ 
for any arithmetic specialization $\lambda$ of $h^\ord$ with $\lambda(\fraka)=0$. 
\end{enumerate}
Consequently, the notion of automorphic types of minimal prime ideals of $h^\ord$ is well-defined. 
\end{mytheorem}

\begin{proof}
Note that $\tr \rho_\fraka$ is a pseudorepresentation of $G_\bbQ$ with values in $\calR(\fraka)$ 
and $\rho_\fraka$ is irreducible. 
For any prime $\frakp$ of $\calR(\fraka)$, the ring $\calR(\fraka)_\frakp$ is 
Noetherian and hence $\calR(\fraka)_\frakp^h$ is Noetherian 
(see \cite[\href{http://stacks.math.columbia.edu/tag/06LJ}{Tag 06LJ}]{StacksProject}). 
So $\calR(\fraka)_\frakp^h$ is Hausdorff 
by Krull intersection theorem (see \cite[Theorem 8.10]{MatsumuraCommutativeRingTheory}). 
Note that by Grothendieck's monodromy theorem (\cite[p.~515--516]{SerreTate}), 
$\rho_\fraka|_{G_\ell}$ is monodromic 
(see the proof of \cite[Lemma 7.8.14]{BellaicheChenevierAsterisQUE}). 
For each arithmetic specialization $\lambda$ of $h^\ord$ with $\lambda(\fraka)=0$, 
$\rho_{f_\lambda}$ is an irreducible $G_F$-representation (by 
\cite[Theorem 2.3]{RibetGalReprAttachedToEigenformsWithNebentypus}) 
over an algebraic closure of the residue field of $\calR(\fraka)_{\frakp_\lambda}$, 
$\tr \rho_{f_\lambda}$ is equal to $\lambda \circ \tr \rho_\fraka$ and 
$\rho_{f_\lambda}|_{G_\ell}$ is pure (by \cite{Carayol86RepresentationAssocieHiblModForm}). 
So by theorem \ref{Thm: purity sum}, we get part (1), (2), (3). 
Since local-global compatibility holds for cusp forms (by \cite{Carayol86RepresentationAssocieHiblModForm}) 
and 
each minimal prime ideal of $h^\ord$ is contained in the kernel of some arithmetic specialization of $h^\ord$ 
(as $h^\ord$ is a finite type $\bbZp[[X]]$-module), the final part follows. 

\end{proof}

\subsubsection{Automorphic representations for definite unitary groups}

Let $F$ be a CM field, $F^+$ be its maximal totally real subfield. 
Let $n\geq 2$ be an integer and assume that if $n$ is even, then $n[F^+:\bbQ]$ is divisible by 4. 
Let $\ell>n$ be a rational prime and assume that every prime of $F^+$ lying above $\ell$  splits in $F$. 
Let $K$ be a finite extension of $\bbQl$ in $\bbQlbar$ which contains the 
image of every embedding $F\hra \bbQlbar$. 
Let $S_\ell$ denote the set of places of $F^+$ above $\ell$. 
Let $R$ denote a finite set of finite places of $F^+$ disjoint from $S_\ell$ and consisting 
of places which split in $F$. 
For each place $v\in S_\ell \cup R$, choose once and for all a place $\widetilde v$ of $F$ lying above $v$. 
For $v\in R$, let 
$\Iw(\widetilde v)$ be the compact open subgroup of $\gln_n(\calO_{F_{\widetilde v}})$ and 
$\chi_v$ be the character as in \cite[\S 2.1, 2.2]{GeraghtyPhDThesis}. 

Let $G$ be the reductive algebraic group over $F^+$ as in \cite[\S 2.1]{GeraghtyPhDThesis}. 
For each dominant weight $\lambda$ (as in \cite[Definition 2.2.3]{GeraghtyPhDThesis}) for $G$, 
the group $G(\bbA_{F^+}^{\infty, R})\times \prod_{v\in R} \Iw(\widetilde v)$ 
acts on the spaces $S_{\lambda, \{\chi_v\}}(\bbQlbar)$, $S_{\lambda, \{\chi_v\}}^\ord(\calO_K)$ 
(as in \cite[Definition 2.2.4, 2.4.2]{GeraghtyPhDThesis}).  
For an irreducible constituent $\pi$ of the 
$G(\bbA_{F^+}^{\infty, R})\times \prod_{v\in R} \Iw(\widetilde v)$-representation $S_{\lambda, \{\chi_v\}}(\bbQlbar)$, 
let $\WBC(\pi)$ denote the weak base change of $\pi$ to $\gln_n(\bbA_F)$ 
(which exists by \cite[Corollaire 5.3]{LabesseChangementDeBaseCMDiscreteSeries}) 
and let $r_\pi:G_F\to \gln_n(\bbQlbar)$ 
(as in \cite[Proposition 2.7.2]{GeraghtyPhDThesis}) 
denote the unique (up to equivalence) continuous semisimple representation 
attached to $\WBC(\pi)$ via \cite[Theorem 3.2.5]{ChenevierHarrisConstructionAutoGalRepr}. 

An irreducible constituent $\pi$ of the 
$G(\bbA_{F^+}^{\infty, R})\times \prod_{v\in R} \Iw(\widetilde v)$-representation 
$S_{\lambda, \{\chi_v\}}(\bbQlbar)$ is said to be an {\it ordinary automorphic 
representation for $G$} if 
$\pi^{U(\frakl^{b,c})}\cap 
S_{\lambda, \{\chi_v\}}^\ord(U(\frakl^{b,c}), \calO_K)
\neq 0$ 
for some integers $0\leq b\leq c$ (see 
\cite[Definition 2.2.4, \S 2.3]{GeraghtyPhDThesis} for details). 
Let $U$ be a compact open subgroup of $G(\bbA_{F^+}^\infty)$, 
$T$ be a finite set of finite places of $F^+$ containing $R\cup S_\ell$ and 
such that every place in $T$ splits in $F$ (see \cite[\S 2.3]{GeraghtyPhDThesis}). 
Let $\bbT^\ord$ denote the universal ordinary Hecke algebra $\bbT^{T, \ord}_{\{\chi_v\}}(U(\frakl^\infty), \calO_K)$
(as in \cite[Definition 2.6.2]{GeraghtyPhDThesis}). 
Let $\Lambda$ be the completed group algebra as in \cite[Definition 2.5.1]{GeraghtyPhDThesis}. 
By definition of $\bbT^\ord$, it is equipped with a $\Lambda$-algebra structure and is finite over $\Lambda$. 
An $\calO_K$-algebra homomorphism $f:A \to \bbQlbar$ is said to 
be an {\it arithmetic specialization} of a finite $\Lambda$-algebra $A$ if 
$\ker (f|_\Lambda)$ is equal to the prime ideal $\wp_{\lambda, \alpha}$ 
(as in \cite[Definition 2.6.3]{GeraghtyPhDThesis}) of $\Lambda$ for 
some dominant weight $\lambda$ for $G$ and a finite order character 
$\alpha:T_n(\frakl) \to \calO_K^\times$. 
By \cite[Lemma 2.6.4]{GeraghtyPhDThesis}, each arithmetic specialization $\eta$ of $\bbT^\ord$ 
determines an ordinary automorphic representation $\pi_\eta$ for $G$. 
An arithmetic specialization $\eta$ of $\bbT^\ord$ is said to be {\it stable} 
if $\WBC(\pi_\eta)$ is cuspidal. 

Let $\frakm$ be a non-Eisenstein maximal ideal of $\bbT^\ord$ (in the sense of \cite[\S 2.7]{GeraghtyPhDThesis}). 
Let $r_\frakm$ denote the representation of $G_{F^+}$ as in \cite[Proposition 2.7.4]{GeraghtyPhDThesis}. 
Then by restricting it to $G_F$ and then composing with the projection 
$\gln_n(\bbT^\ord_\frakm)\times \gln_1(\bbT^\ord_\frakm)\to \gln_n(\bbT^\ord_\frakm)$, we get a continuous 
representation $G_F \to \gln_n(\bbT^\ord_\frakm)$ which is denoted by $r_\frakm$ by abuse of notation. 
Since $\frakm$ is non-Eisenstein, the $G_F$-representations 
$\eta\circ r_\frakm$ and $r_{\pi_\eta}$ are isomorphic 
for any arithmetic specialization $\eta$ of $\bbT^\ord_\frakm$
(by \cite[Proposition 2.7.2, 2.7.4]{GeraghtyPhDThesis}).

\begin{mydefinition}
\label{Defn: automorphic type: unitary}
Let $w$ be a finite place of $F$ not lying above $\ell$ and 
$\fraka$ be a minimal prime of $\bbT^\ord$. 
If the maximal ideal of $\bbT^\ord$ containing $\fraka$ is non-Eisenstein 
and some stable arithmetic specialization of $\bbT^\ord$ vanishes on $\fraka$, 
then the \textnormal{automorphic type} of $\fraka$ at $w$ is defined to be the unordered tuple 
$\AT_w(\fraka)$ if the automorphic types of $\WBC(\pi_\eta)_w$ are equal to $\AT_w(\fraka)$ 
for all stable arithmetic specialization $\eta$ of $\bbT^\ord$ with $\eta(\fraka)=0$. 
\end{mydefinition}

\begin{mytheorem}
\label{Thm: Application: Hida unitary}
Let $w\nmid \ell$ be a finite place of $F$, $\fraka$ be a 
minimal prime of $\bbT^\ord_\frakm$ and $\frakm$ be the maximal ideal of $\bbT^\ord$ 
containing $\fraka$. Suppose $\frakm$ is non-Eisenstein.  
Denote the quotient ring $\bbT^\ord_\frakm/\fraka$ by $\calR(\fraka)$ and 
the representation $r_\frakm \modu \fraka$ by $r_\fraka$. 
Then there exist positive integers $m, t_1\leq \cdots \leq t_m$ and 
irreducible Frobenius-semisimple representations $r_1, \cdots, r_m$ of $W_w$ 
over $\calR(\fraka)^\intal[1/\ell]$ 
such that 
$\WD(r_\fraka|_{W_w})^\Frss$ is isomorphic to $\oplus_{i=1}^m \Sp_{t_i}(r_i)$ 
over $\Qbar(\calR(\fraka))$ 
and 
$\WD(r_{\pi_\eta}|_{W_w})^\Frss$ is isomorphic to 
$\oplus_{i=1}^m \Sp_{t_i}(\eta^\intal \circ r_i)$ 
for any stable arithmetic specialization $\eta$ of $\calR(\fraka)$. 
Consequently, the notion of local automorphic types of minimal prime ideals of $\bbT^\ord$ is well-defined. Moreover, 
two minimal prime ideals of $\bbT^\ord$ are contained in two non-Eisenstein maximal 
ideals and are both contained in the kernel of a stable arithmetic specialization 
of $\bbT^\ord$ only if their automorphic types at any finite place $v\nmid \ell$ of $F$ are the same. 
\end{mytheorem}
\begin{proof}
If $\pi$ is an irreducible constituent of the 
$G(\bbA_{F^+}^{\infty, R})\times \prod_{v\in R} \Iw(\widetilde v)$-representation 
$S_{\lambda, \{\chi_v\}}(\bbQlbar)$ such that $\WBC(\pi)$ is cuspidal, 
then for any finite place $w$ of $F$ not dividing $\ell$, 
$r_\pi|_{G_w}$ is pure by 
\cite[Theorem 1.1, 1.2]{CaraianiLocalGlobalCompatibility} 
and the proofs of theorem 5.8, corollary 5.9 of \loccit. 
Note that $r_\fraka|_{W_w}$ is monodromic by Grothendieck's 
monodromy theorem (see \cite[p.\,515--516]{SerreTate}). 
So theorem \ref{Thm: purity big} (or theorem \ref{Thm: purity sum}) gives the first part. 
By \cite[Theorem 1.1]{CaraianiLocalGlobalCompatibility} on local-global compatibility 
of cuspidal automorphic representations for $\gln_n$, the notion of local automorphic 
types is well-defined. Then the rest follows from the first part. 
\end{proof}

\subsection{Eigenvarieties}

Let $X$ be a rigid analytic space over a finite extension of $\bbQp$. 
If $z$ is an element of $X(\bbQpbar)$, then 
the map $\calO(X)\to \bbQpbar$ is denoted by 
$\ev_{zX}$. 
The restriction map between the global sections of two admissible open subsets $U\supset V$ of $X$ is denoted by $\res_{UV}$.

Let $E/\bbQ$ be an imaginary quadratic field and $G$ denote the definite unitary group 
$U(m)$ (as in \cite[\S 6.2.2]{BellaicheChenevierAsterisQUE}) in $m\geq 1$ variables. 
We assume that $p$ splits in $E$. 
Let $\calH$ denote the Hecke algebra as in \cite[\S 7.2.1]{BellaicheChenevierAsterisQUE}. 
Let $\calZ_0 \subset \Hom_{\ring}(\calH, \bbQpbar)\times \bbZ^m$ be the set of pairs 
$(\psi_{(\pi, \calR)}, \underline k)$ associated to the $p$-refined 
automorphic representations $(\pi, \calR)$ of any weight $\underline k$ (see 
\cite[\S 7.2.2, 7.2.3]{BellaicheChenevierAsterisQUE}). 
Let $e$ be the idempotent as in \cite[\S 7.3.1]{BellaicheChenevierAsterisQUE} and 
let $\calZ_e\subset \calZ_0$ denote the subset of $(\psi_{(\pi, \calR)}, \underline k)$ such 
that $e(\pi^p)\neq 0$. We assume that $\calZ_e$ is nonempty. 
Then by \cite[\S 7.3]{BellaicheChenevierAsterisQUE}, there exists an eigenvariety for $\calZ_e$, 
\ie, there exist a 
reduced rigid analytic space $X$ over a finite extension $L$ of $\bbQp$, 
a ring homomorphism $\psi:\calH \to \calO(X)$, 
an analytic map $\omega:X \to \Hom ( (\bbZ_p^\times)^m, \bbG_m^\rig) \times _\bbQp L $ 
and 
an accumulation and Zariski-dense subset $Z$ of $X(\bbQpbar)$ 
such that conditions (i), (ii), (iii) of \cite[Definition 7.2.5]{BellaicheChenevierAsterisQUE} hold. 
In particular, $z\mapsto (\ev_{zX} \circ \psi, \omega(z))$ induces a bijection $Z\xra{\sim} \calZ_e$. 
The set $Z$ is called the set of {\it arithmetic points} of $X$. 
Let $Z_\reg\subset Z$ be the subset of points parametrizing the $p$-refined automorphic representations $(\pi, \calR)$ 
such that $\pi_\infty$ is regular and the semisimple conjugacy class of $\pi_p$ has $m$ distinct eigenvalues 
(see \cite[\S 7.5.1]{BellaicheChenevierAsterisQUE}). 
By \cite[Lemma 7.5.3]{BellaicheChenevierAsterisQUE}, 
$Z_\reg$ is a Zariski-dense subset of $X$. 
For each $z\in Z$, we fix a $p$-refined automorphic representation $\pi_z$ of $U(m)$ such that $z$ corresponds to $\pi_z$ under the bijection $Z\xra{\sim}\calZ_e$. 
For each $z\in Z_\reg$, let $\rho_{z, p}:G_{E} \to \gln_m(\bbQpbar)$ 
denote the unique (up to equivalence) continuous semisimple representation 
attached to $\WBC(\pi_z)$ via \cite[Theorem 3.2.5]{ChenevierHarrisConstructionAutoGalRepr}. 
By \cite[Proposition 7.5.4]{BellaicheChenevierAsterisQUE}, 
there exists a pseudorepresentation $T:G_{E} \to \calO(X)$ such that 
$\ev_{zX} \circ T= \tr \rho_{z, p}$ for all $z\in Z_\reg$. 
Let $Z_\reg^\st$ denote the set of points $z\in Z_\reg$ such that 
$\WBC(\pi_z)$ is cuspidal. 
For $z\in Z_\reg^\st$, the Galois representation $\rho_{z, p}$ is 
expected to be irreducible. It is known when $m\leq 3$ by \cite{BlasiusRogawskiTwoBall} 
and in many cases when $m=4$ by an unpublished work of Ramakrishnan. 
By \cite[Theorem D]{PatrikisTaylor}, it is known for infinitely many primes $p$. 

\begin{mydefinition}
\label{Defn: automorphic type: eigen}
Let $w$ be a finite place of $E$ not lying above $p$ and 
$Y_0$ be an irreducible component of $X$ such that $Z_\reg^\st \cap Y_0$ is nonempty. 
The \textnormal{automorphic type} of $Y_0$ at $w$ is defined to be the unordered tuple 
$\AT_w(Y_0)$ if the automorphic types of $\WBC(\pi_z)_w$ are equal to $\AT_w(Y_0)$ 
for all $z\in Z_\reg^\st\cap Y_0$. 
\end{mydefinition}

Let $\xi:\widetilde X \to X$ be a normalization of $X$. 
Let $C$ be a connected component of $\widetilde X$ and $Y$ be the irreducible 
component $\xi(C)$ (together with its canonical structure of reduced rigid space) of $X$. 
By \cite[Lemma 2.2.1 (2)]{ConradIrreducibleComponentsOfRigidSpaces}, the map $\xi|_C:C \to Y$ is a normalization. 
For each $x\in X(\bbQpbar)\cap Y$, we fix a point $\widetilde x$ in $C(\bbQpbar)$ which goes to $x$ under the map 
$C(\bbQpbar) \to Y(\bbQpbar)$.

\begin{mytheorem}
\label{Thm: Application: Eigen}
Let $w\nmid p$ be a finite place of $E$. 
Suppose that the intersection of $Z_\reg^\st$ with any irreducible component of $X$ is nonempty 
and for any $z\in Z_\reg^\st$, the Galois representation $\rho_{z, p}$ is irreducible.  
Then there exist positive integers $n, t_1, \cdots, t_n$, 
irreducible Frobenius-semisimple representations 
$r_1, \cdots, r_n$ of $W_w$ 
over $\calO(C)^\intal$ such that the following hold. 
\begin{enumerate}
\item The polynomial $(\Eul (\rho_C, N_C))^\mo$ has coefficients in $\calO(C)^\intal$ 
and $\res_{\widetilde XC} \circ \xi \circ T|_{W_w}$ is equal to the trace of $(\rho_C, N_C)$ where 
$$(\rho_C, N_C):=\bigoplus_{i=1}^n \Sp_{t_i}(r_i)_{/\calO(C)^\intal}.$$

\item If $z\in Z_\reg^\st\cap Y$, or more generally if $z\in Z_\reg \cap Y$ such that 
$\rho_{z, p}$ is irreducible and $\rho_{z, p}|_{W_w}$ is pure,  
then for any arbitrary lift $\ev_{\widetilde z C}^\intal$ of $\ev_{\widetilde z C}$, there is 
an isomorphism
$$\WD(\rho_{z, p}|_{W_w})^\Frss \simeq
\ev^\intal_{\widetilde zC} \circ (\rho_C, N_C)
$$
and 
$$\ev_{\widetilde zC}^\intal (\Eul(\rho_C, N_C))=\Eul_w(\rho_{\pi_z,p}).$$

\item Let $V$ be a nonempty connected admissible open subset of $C$ and 
$(\rho_V, N_V): W_w \to \gln_m(\calO(V)^\intal)$ be a Weil-Deligne representation 
such that $\res_{CV} \circ \res_{\widetilde X C} \circ \xi \circ T=\tr \rho_V$ and 
$f_V\circ (\rho_V, N_V)$ is pure for some $\bbZp$-algebra homomorphism $f_V: \calO(V)^\intal \to \bbQpbar$. 
Then for any arbitrary lift $\res_{CV}^\intal$ of $\res_{CV}$, there is an isomorphism 
$$((\rho_V, N_V)\otimes_{\calO(V)} 
\overline Q (\calO(V))
 )^\Frss
\simeq 
(\res^\intal_{CV} \circ (\rho_C, N_C))
\otimes_{\calO(V)^\intal }
\overline Q (\calO(V)).$$
\end{enumerate}
Consequently, the notion of local automorphic types of irreducible components of $X$ is well-defined. 
Moreover, two irreducible components of $X$ intersect at 
a point of $Z_\reg^\st$ only if their local automorphic types 
at any finite place of $E$ outside $p$ are the same. 
\end{mytheorem}
\begin{proof}
By \cite[Lemma 2.1.4]{ConradIrreducibleComponentsOfRigidSpaces}, $\calO(C)$ is an integral domain over $\bbZp$. 
By \cite[Theorem 1]{TaylorGaloisReprAssociatedToSiegelModForms}, 
$\res_{\widetilde X C} \circ \xi \circ T$ is equal to the trace of a representation $\sigma$ of $G_E$ 
over $\Qbar (\calO(C))$. Since $Z_\reg^\st\cap Y$ is nonempty and the Galois representations attached 
to its points are irreducible, the representation $\sigma$ is irreducible. 
Let $z$ be a point in $Z_\reg^\st\cap Y$. 
We choose a connected affinoid neighbourhood $U_{\widetilde z}$ of $\widetilde z$. 
Note that $U_{\widetilde z}$ is contained in $C$ 
and the map $\res_{C U_{\widetilde z}}$ is injective 
by \cite[Lemma 2.1.4]{ConradIrreducibleComponentsOfRigidSpaces}. 
The point $\widetilde z$ defines a maximal ideal $\frakm_{\widetilde z}$ of $\calO(U_{\widetilde z})$. 
The localization of $\calO(U_{\widetilde z})$ at $\frakm_{\widetilde z}$ 
is Noetherian and hence the Henselization of this localization is Hausdorff 
by Krull intersection theorem (see \cite[Theorem 8.10]{MatsumuraCommutativeRingTheory}). 
Note that $\rho_{z, p}$ is an irreducible representation of $G_E$ over an algebraic closure of the residue 
field of $\calO(U_{\widetilde z})_{\frakm_{\widetilde z}}$ 
and has trace equal to $\pi_{\frakm_{\widetilde z}} \circ \res_{CU_{\widetilde z}} \circ (\res_{\widetilde X C} \circ \xi \circ T)$. 
By \cite[Theorem 1]{TaylorGaloisReprAssociatedToSiegelModForms}, 
there exists a semisimple $G_E$-representation 
$\widetilde \rho_{U_{\widetilde z}}$ over $\Qbar (\calO(U_{\widetilde z}))$ 
such that $\tr \widetilde \rho_{U_{\widetilde z}}= \res_{C U_{\widetilde z}} \circ \res_{\widetilde X C} \circ \xi \circ T$ 
and the restriction of $\widetilde \rho_{U_{\widetilde z}}$ to $W_w$ is monodromic 
by \cite[Lemma 7.8.11, 7.8.14]{BellaicheChenevierAsterisQUE}. 
So $\sigma|_{W_w}$ is monodromic. 
Since $\rho_{z, p}$ is irreducible, $\widetilde \rho_{U_{\widetilde z}}$ is also irreducible. 
By Zorn's lemma, we have 
$\ev^\intal_{\widetilde z C}= \ev_{\widetilde z U_{\widetilde z}}^\dag \circ \res_{CU_{\widetilde z}}^\dag$ for some lifts  
$\ev_{\widetilde z U_{\widetilde z}}^\dag, \res_{CU_{\widetilde z}}^\dag$ of $\ev_{\widetilde z U_{\widetilde z}}, \res_{CU_{\widetilde z}}$ 
respectively. 
So by theorem \ref{Thm: purity sum}, we get part (1) and (2). 
Using theorem \ref{Thm: purity for pseudorepresentations}, we get part (3). 
By \cite[Theorem 1.1]{CaraianiLocalGlobalCompatibility} on local-global compatibility 
of cuspidal automorphic representations for $\gln_n$, the notion of local automorphic 
types is well-defined. Then the rest follows. 
\end{proof}

\subsection*{Acknowledgements}
This work would have been impossible without the support and guidance of Olivier Fouquet. I thank him for carefully going through the earlier versions of this manuscript and making valuable comments. I also thank Adrian Iovita for providing a reference. A part of this work was completed as a component of my thesis and I would like to acknowledge the financial support provided by the Algant consortium during my doctoral studies at Universit\'e Paris-Sud and Universit\`a degli Studi di Padova under the supervision of Olivier Fouquet and Adrian Iovita. I also acknowledge the financial support from the ANR Projet Blanc ANR-10-BLAN 0114. I am indebted to Jo\"el Bella\"iche and Laurent Clozel for suggesting corrections and improvements. During the preparation of this article, I benefited from discussions with Yiwen Ding, Santosh Nadimpalli.

\providecommand{\bysame}{\leavevmode\hbox to3em{\hrulefill}\thinspace}
\providecommand{\MR}{\relax\ifhmode\unskip\space\fi MR }
\providecommand{\MRhref}[2]{%
  \href{http://www.ams.org/mathscinet-getitem?mr=#1}{#2}
}
\providecommand{\href}[2]{#2}

\end{document}